\documentclass[12pt]{article}
\usepackage{amsfonts,amssymb,amsmath,amsthm}
\usepackage[russian,english]{babel}
\date{}

\theoremstyle{plain}
\newtheorem{theorem}{Theorem}
\newtheorem{proposition}{Proposition}
\newtheorem{lemma}{Lemma}
\newtheorem{corollary}{Corollary}
\theoremstyle{definition}
\newtheorem{definition}{Definition}
\newtheorem{remark}{Remark}
\numberwithin{equation}{section}
\numberwithin{theorem}{section}
\numberwithin{proposition}{section}
\numberwithin{lemma}{section}
\numberwithin{corollary}{section}
\numberwithin{definition}{section}
\numberwithin{remark}{section}
\newcommand{\R}{\mathbb{R}}
\newcommand{\N}{\mathbb{N}}
\newcommand{\Z}{\mathbb{Z}}
\renewcommand{\C}{\mathbb{C}}
\newcommand{\D}{\mathcal{D}}
\newcommand{\T}{\mathbb{T}}
\newcommand{\B}{{\mathcal B}}
\newcommand{\esslim}{\operatornamewithlimits{ess\,lim}}

\newcommand{\sgn}{\operatorname{sign}}
\newcommand{\meas}{\operatorname{meas}}

\newcommand{\Var}{\operatorname{Var}}
\newcommand{\pr}{\operatorname{pr}}
\newcommand{\<}{\langle}
\renewcommand{\>}{\rangle}
\newcommand{\MV}{\operatorname{MV}}
\newcommand{\const}{\mathrm{const}}
\renewcommand{\div}{\operatorname{div}}
\newcommand{\supp}{\operatorname{supp}}

\def\Xint#1{\mathchoice
{\XXint\displaystyle\textstyle{#1}}%
{\XXint\textstyle\scriptstyle{#1}}%
{\XXint\scriptstyle\scriptscriptstyle{#1}}%
{\XXint\scriptscriptstyle\scriptscriptstyle{#1}}%
\!\int}
\def\XXint#1#2#3{{\setbox0=\hbox{$#1{#2#3}{\int}$ }
\vcenter{\hbox{$#2#3$ }}\kern-.57\wd0}}

\def\dashint{\Xint-}

\begin{document}
\title{Decay of periodic entropy solutions to degenerate nonlinear parabolic equations}

\author{Evgeniy Yu. Panov}
\maketitle
\begin{abstract}
Under a precise nonlinearity-diffusivity condition we establish the decay of space-periodic entropy solutions of a multidimensional degenerate nonlinear parabolic equation.
\end{abstract}

\section{Introduction}
In the half-space $\Pi=\R_+\times\R^n$, $\R_+=(0,+\infty)$, we consider the nonlinear parabolic equation
\begin{equation}\label{1}
u_t+\div_x(\varphi(u)-a(u)\nabla_x u)=0,
\end{equation}
where the flux vector $\varphi(u)=(\varphi_1(u),\ldots,\varphi_n(u))$ is merely continuous: $\varphi_i(u)\in C(\R)$, $i=1,\ldots,n$, and the diffusion matrix $a(u)=(a_{ij}(u))_{i,j=1}^n$ is Lebesgue measurable and bounded:
$a_{ij}(u)\in L^\infty(\R)$, $i,j=1,\ldots,n$. We also assume that the matrix ${a(u)\ge 0}$
(nonnegative definite). This matrix may have nontrivial kernel. Hence (\ref{1}) is a degenerate
(hyperbolic-parabolic) equation. In particular case $a\equiv 0$ it reduces to a first order conservation law
 \begin{equation}\label{con}
u_t+\div_x \varphi(u)=0.
\end{equation}
Equation (\ref{1}) is endowed with the initial condition
\begin{equation}\label{2}
u(0,x)=u_0(x).
\end{equation}

Let $g(u)\in BV_{loc}(\R)$ be a function of bounded variation on any segment in $\R$.
We will need the bounded linear operator $T_g:C(\R)/C\to C(\R)/C$, where $C$ is the space of constants. This operator is defined up to an additive constant by the relation
\begin{equation}\label{efl}
T_g(f)(u)=g(u-)f(u)-\int_0^u f(s)dg(s),
\end{equation}
where
$\displaystyle g(u-)=\lim_{v\to u-} g(v)$ is the left limit of $g$ at the point $u$, and the integral in (\ref{efl}) is understood in accordance with the formula
$$
\int_0^u f(s)dg(s)=\sgn u\int_{J(u)} f(s)dg(s),
$$
where $\sgn u=1$,
$J(u)$ is the interval $[0,u)$ if $u>0$, and $\sgn u=-1$, $J(u)=[u,0)$ if $u\le 0$.
Observe that $T_g(f)(u)$ is continuous even in the case of discontinuous $g(u)$. For instance, if $g(u)=\sgn(u-k)$ then
$T_g(f)(u)=\sgn(u-k)(f(u)-f(k))$. Notice also that for $f\in C^1(\R)$ the operator $T_g$ is uniquely determined by the identity $T_g(f)'(u)=g(u)f'(u)$ (in $\D'(\R)$).
In the case $f'(u),g(u)\in L^2_{loc}(\R)$ the function $T_g(f)\in C(\R)/C$ can be defined by the identity $T_g(f)'(u)=g(u)f'(u)\in L^1_{loc}(\R)$. As is easy to see, in this case
the correspondence $g\to T_g(f)$ is a linear continuous map from $L^2_{loc}(\R)$ into $C(\R)/C$.

We fix some representation of the diffusion matrix $a(u)$ in the form $a(u)=b^\top(u)b(u)$, where $b(u)=(b_{ij}(u))_{i,j=1}^n$ is matrix-valued function with measurable and bounded entries,  $b_{ij}(u)\in L^\infty(\R)$.
We recall the notion of entropy solution of the Cauchy problem (\ref{1}), (\ref{2}) introduced in \cite{ChPer1}.

\begin{definition}\label{def1}
A function $u=u(t,x)\in L^\infty(\Pi)$ is called an entropy solution (e.s. for short) of (\ref{1}), (\ref{2}) if
the following conditions hold:

(i) for each $r=1,\ldots,n$ the distributions
\begin{equation}\label{pr}
\div_x B_r(u(t,x))\in L^2_{loc}(\Pi),
\end{equation}
where vectors $B_r(u)=(B_{r1}(u),\ldots,B_{rn}(u))\in C(\R,\R^n)$, and $B_{ri}'(u)=b_{ri}(u)$, $r,i=1,\ldots,n$;

(ii) for every $g(u)\in C^1(R)$, $r=1,\ldots,n$
\begin{equation}\label{cr}
\div_x T_g(B_r)(u(t,x))=g(u(t,x))\div_x B_r(u(t,x)) \ \mbox{ in } \D'(\Pi);
\end{equation}

(iii) for any convex function $\eta(u)\in C^2(\R)$
\begin{eqnarray}\label{entr}
\eta(u)_t+\div_x(T_{\eta'}(\varphi)(u))-D^2\cdot T_{\eta'}(A(u))+ \nonumber\\ \eta''(u)\sum_{r=1}^n (\div_x B_r(u))^2\le 0 \ \mbox{ in }
\D'(\Pi);
\end{eqnarray}

(iv) $\displaystyle\esslim_{t\to 0} u(t,\cdot)=u_0$ in $L^1_{loc}(\R^n)$.
\end{definition}
In (\ref{entr}) the operator
$$
D^2\cdot T_{\eta'}(A(u))\doteq\sum_{i,j=1}^n \frac{\partial^2}{\partial x_i\partial x_j}T_{\eta'}(A_{ij}(u)).
$$
Relation (\ref{entr}) means that for any non-negative test function $f=f(t,x)\in C_0^\infty(\Pi)$
\begin{eqnarray*}
\int_\Pi [\eta(u)f_t+T_{\eta'}(\varphi)(u)\cdot\nabla_x f+T_{\eta'}(A(u))\cdot D^2_x f]dtdx\ge 0,
\end{eqnarray*}
where $D^2_x f$ is the symmetric matrix of second order derivatives of $f$, and "$\cdot$" denotes the standard scalar multiplications of vectors or matrices.

\begin{remark}\label{rem0}
The chain rule postulated in (ii) actually holds for arbitrary locally bounded Borel function $g(u)\in L^\infty_{loc}(\R)$.
First of all, observe that since $B'_r(u)\in L^\infty(\R,\R^n)$ then
\begin{equation}\label{ch1}
T_g(B_r)(u)=\int_0^u g(v)B'_r(v)dv.
\end{equation}
For $R>0$ we consider the set
$F_R$ consisting of bounded Borel functions $g(u)$ such that $\sup |g(u)|\le R$ and the chain rule (\ref{cr}) holds. By (ii)
$F_R$ contains all bounded functions $g(u)\in C^1(\R)$ such that $\sup |g(u)|\le R$. Let $g_l(u)\in F_R$, $l\in\N$, be a sequence which converges pointwise to a function $g(u)$ as $l\to\infty$. It is clear that $g(u)$ is a Borel function and $\sup |g(u)|\le R$. Observe that the functions $g_l(u(t,x))$ are bounded and measurable (because $g_l$ are Borel functions), the sequence $g_l(u(t,x))\to g(u(t,x))$ pointwise as $l\to\infty$, and $\|g_l(u(t,x))\|_\infty\le R$. This implies that
$$
g_l(u(t,x))\div_x B_r(u(t,x))\mathop{\to}_{l\to\infty} g(u(t,x))\div_x B_r(u(t,x)) \ \mbox{ in } L^2_{loc}(\Pi).
$$
From the other hand,
it follows from (\ref{ch1}) that $T_{g_l}(B_r)(u(t,x))$ converges to $T_g(B_r)(u(t,x))$ as $l\to\infty$ uniformly on $\Pi$. Therefore, we can pass to the limit as $l\to\infty$ in relation (\ref{cr}) with $g=g_l$ and derive that (\ref{cr}) holds for our limit function $g(u)$. Thus, $g(u)\in F_R$. We see that $F_R$ is closed under pointwise convergence and by Lebesgue theorem  $F_R$ contains all Borel functions $g$ such that $\sup |g|\le R$. Since $R>0$ is arbitrary, we find that (\ref{cr}) holds for all bounded Borel functions. It only remains to notice that the behavior of $g(u)$ out of the segment $[-M,M]$, where $M=\|u\|_\infty$, does not matter. Therefore, (\ref{cr}) holds for any Borel function bounded on $[-M,M]$, in particular, for each locally bounded Borel function.

Remark also that for correctness of (\ref{cr}) we have to choose the Borel representative for $g(u)$ (recall that $g(u)$ is defined up to equality on a set of full measure and such the representative exists). Notice that $T_g(B_r)(u)$ does not depend on the choice of a Borel representative of $g$. Hence, the right-hand side of (\ref{cr}) does not depend on this choice either.
\end{remark}

If to be precise, in \cite{ChPer1} the representation $a=b^\top b$ was used with $b=a(u)^{1/2}$. In order to make the definition invariant under linear changes of the independent variables, we have to extend the class of admissible representations. For instance, let us introduce the change $y=y(t,x)=qx-tc$, where $q:\R^n\to\R^n$ is a nondegenerate linear map and $c\in\R^n$. This can be written in the coordinate form as
\begin{equation}\label{ch}
y_i=\sum_{j=1}^nq_{ij}x_j-c_it, \quad i=1,\ldots,n.
\end{equation}
As is easy to verify, the function $u=v(t,y(t,x))$ is an e.s. of (\ref{1}), (\ref{2}) with initial data $u_0=v_0(y(0,x))$ if and only if $v(t,y)$ is an e.s. of the problem
$$
v_t+\div_y(\tilde\varphi(v)-\tilde a(v)\nabla_y v)=0, \quad v(0,y)=v_0(y),
$$
where
\begin{equation}\label{after}
\tilde\varphi(v)=q\varphi(v)-cv, \quad \tilde a(v)=qa(v)q^\top,
\end{equation}
corresponding to the representation $\tilde a(v) =(b(v)q^\top)^\top(b(v)q^\top)$, $b(v)q^\top\in L^\infty(\R)$, of the diffusion matrix $\tilde a(v)$. We remark that
$$
\div_y (B(v)q^\top)_r\big|_{y=y(t,x)}=\sum_{l,j=1}^n \partial_{y_l}B_{rj}(v)q_{lj}=\sum_{j=1}^n \partial_{x_j}B_{rj}(u)=\div_x B_r(u).
$$
Recall that $B_r'(u)=b_r(u)=(b_{r1}(u),\ldots,b_{rn}(u))$, $r=1,\ldots,n$.

We underline that the matrix $bq^\top$ is not necessarily symmetric and therefore differs from $\tilde a(v)^{1/2}$.

\medskip
Now, we suppose that the initial function $u_0$ is periodic with a lattice of periods $L$: $u_0(x+e)=u_0(x)$
a.e. in $\R^n$ for all $e\in L$. Denote $\T^n=\R^n/L$ the corresponding torus (which can be identified with a fundamental parallelepiped), $dx$ is the Lebesgue measure on $\T^n$, normalized by the condition $dx(\T^n)=V$, $V$ being volume of a fundamental parallelepiped (observe that $V$ does not depend on its choice). Then there exists a space-periodic e.s. $u=u(t,x)$ of the problem (\ref{1}), (\ref{2}), $u(t,x+e)=u(t,x)$ a.e. in $\Pi$ for all $e\in L$. This can be written as $u\in L^\infty((0,+\infty)\times\T^n)$. This solution can be constructed as a limit of the sequence $u_k$ of solutions to the regularized problem
$$
u_t+\div_x(\varphi_k(u)-a_k(u)\nabla_x u)=0,
$$
with smooth flux vectors $\varphi_k(u)$, and smooth and strictly positive definite diffusion matrices $a_k(u)$, which approximate $\varphi(u)$ and $a(u)$, respectively. For details, we refer to \cite{ChPer1,ChKarl}. It is known \cite[Theorem 1.2]{ChPer1} that e.s. $u(t,x)$ satisfies the maximum principle: $|u(t,x)|\le \|u_0\|_\infty$ for a.e. $(t,x)\in\Pi$.

Our main result is the long time decay property of entropy solutions. Let
$$
L'=\{\xi\in\R^n | \xi\cdot e\in\Z \ \forall e\in L \}
$$
be a dual lattice to the lattice of periods $L$, $$I=\frac{1}{V}\int_{\T^n} u_0(x)dx$$ be the mean value of initial data.

\begin{theorem}\label{thM}
Assume that the following \textbf{nonlinearity-diffusivity condition} holds: for all $\xi\in L'$, $\xi\not=0$ there is no vicinity of $I$, where simultaneously the function $\xi\cdot\varphi(u)$ is affine and the function $a(u)\xi\cdot\xi\equiv 0$. Then
\begin{equation}\label{dec}
\esslim_{t\to+\infty} u(t,x)=I \ \mbox{ in } L^1(\T^n).
\end{equation}
\end{theorem}

The decay property was established in \cite{ChPer2} under the more restrictive version of the nonlinearity-diffusivity condition, in the case of locally Lipschitz flux vector when $\varphi'(u)\in L^\infty_{loc}(\R,\R^n)$. This conditions reads: for all $\tau\in\R$, $\xi\in L'$, $\xi\not=0$
\begin{equation}\label{nd}
\meas\{ \ u\in\R, |u|\le\|u_0\|_\infty \ | \ \tau u+\xi\cdot\varphi'(u)=a(u)\xi\cdot\xi=0 \ \}=0.
\end{equation}
In the hyperbolic case $a\equiv 0$
decay property (\ref{dec}) was proved in \cite{PaNHM}, see also the previous papers \cite{ChF,Daf,PaAIHP1}.

Let us show that our nonlinearity-diffusivity condition is exact. In fact, if it fails, we can find an interval $(I-\delta,I+\delta)$, a nonzero vector $\xi\in L'$, and a constant $c\in\R$ such that $\xi\cdot\varphi(u)-cu\equiv\const$, $a(u)\xi\cdot\xi\equiv 0$ on this interval. Then, as is easily verified, the function $u(t,x)=I+\delta\sin(2\pi(\xi\cdot x-ct))$ is an $x$-periodic (with the lattice of periods $L$) e.s. of (\ref{1}), (\ref{2}) with the periodic initial data $u_0(x)=I+\delta\sin(2\pi(\xi\cdot x))$, which has the mean value $I$. It is clear that $u(t,x)$ does not converges as $t\to+\infty$ in $L^1(\T^n)$, and the decay property fails.

\section{Some properties of periodic entropy solutions}

\begin{lemma}\label{lemA1} For each convex $\eta(u)$ there exists a locally finite nonnegative $x$-periodic Borel measure $\mu_\eta$ on $\Pi$ such that
\begin{equation}\label{distr1}
\eta(u)_t+\div_x(T_{\eta'}(\varphi)(u))-D^2\cdot T_{\eta'}(A(u))=-\mu_\eta \ \mbox{ in } \D'(\Pi).
\end{equation}
This measure can be identified with a finite measure on $\R_+\times\T^n$. Moreover, for almost each (a.e.) $t_1,t_2$, $0<t_1<t_2$,
\begin{equation}\label{est1}
\mu((t_1,t_2)\times\T^n)\le \int_{\T^n} \eta(u(t_1,x))dx-\int_{\T^n} \eta(u(t_2,x))dx.
\end{equation}
\end{lemma}

\begin{proof}
It follows from (\ref{entr}) that for every convex $\eta(u)$ the distribution
$$
\eta(u)_t+\div_x(T_{\eta'}(\varphi)(u))-D^2\cdot T_{\eta'}(A(u))\le 0 \ \mbox{ in } \D'(\Pi).
$$
By the Schwartz theorem on the representation of nonnegative distributions we conclude that (\ref{distr1}) holds for some  locally finite nonnegative Borel measure $\mu_\eta$ on $\Pi$. Ii is clear that $\mu_\eta$ is periodic with respect to $x$ and can be treated as a measure on $\R_+\times\T^n$. Let $\alpha(t)\in C_0^1(\R_+)$, $\beta(y)\in C_0^2(\R^n)$, $\alpha(t),\beta(y)\ge 0$, $\displaystyle\int_{\R^n}\beta(y)dy=1$. Applying (\ref{distr1}) to the test function $\alpha(t)\beta(x/k)$, with $k\in\N$, we arrive at the relation
\begin{eqnarray*}
\<\mu_\eta,\alpha(t)\beta(x/k)\>\doteq\int_\Pi\alpha(t)\beta(x/k)d\mu_\eta(t,x) =\int_{\Pi}\eta(u(t,x))\alpha'(t)\beta(x/k)dtdx+ \\
\frac{1}{k}\int_{\Pi}T_{\eta'}(\varphi)(u)\cdot\nabla_y\beta(x/k)\alpha(t)dtdx+  \frac{1}{k^2}\int_{\Pi}T_{\eta'}(A)(u)\cdot D^2_y\beta(x/k)\alpha(t)dtdx.
\end{eqnarray*}
Multiplying this equality by $k^{-n}$ and passing to the limit as $k\to\infty$, we obtain
\begin{equation}\label{lm1}
\int_{\R_+\times\T^n}\alpha(t)d\mu_\eta(t,x)=\int_{\R_+\times\T^n}\eta(u(t,x))\alpha'(t)dtdx,
\end{equation}
where we use the known property
$$
\lim_{k\to\infty}k^{-n}\<\mu,\alpha(t)\beta(x/k)\>=\int_{\R_+\times\T^n}\alpha(t)d\mu(t,x)\int_{\R^n}\beta(y)dy
$$
for an arbitrary $x$-periodic locally finite measure on $\Pi$. Identity (\ref{lm1}) means that
$$
\frac{d}{dt}\int_{\T^n}\eta(u(t,x))dx=-\int_{\T^n}d\mu_\eta(t,x) \ \mbox{ in } \D'(\R^+),
$$
where $\displaystyle\int_{\T^n}d\mu_\eta(t,x)$ is treated as the measure $\mu_\eta^t$ on $\R_+$ defined by the relation
$$\<\mu_\eta^t,\alpha(t)\>=\int_{\R_+\times\T^n}\alpha(t)d\mu_\eta(t,x),
$$
that is, $\mu_\eta^t$ is the projection of $\mu_\eta$ on the $t$-axis. It follows from (\ref{lm1}) that for all $t_1,t_2\in\R_+$, $t_1<t_2$, being Lebesgue points of the function
$\displaystyle t\to\int_{\T^n}\eta(u(t,x))dx$ relation (\ref{est1}) holds.
\end{proof}

\begin{corollary}\label{cor0} Let $\eta(u)$ be a convex function. Then

(i) The function $I_\eta(t)=\displaystyle\int_{\T^n}\eta(u(t,x))dx$ decreases: $I_\eta(t_2)\le I_\eta(t_1)$ for a.e. $t_1,t_2\ge 0$, $t_2>t_1$;

(ii) for a.e. $t>0$
\begin{equation}\label{mass}
\frac{1}{V}\int_{\T^n}u(t,x)dx=I\doteq\frac{1}{V}\int_{\T^n}u_0(x)dx;
\end{equation}

(iii) the measure $\mu_\eta$ is finite on $\R_+\times\T^n$. Moreover,
$$
\mu_\eta(\R_+\times\T^n)=\int_{\T^n}\eta(u_0(x))dx-\esslim_{t\to\infty}\int_{\T^n}\eta(u(t,x))dx\le
2V\max_{|r|\le M} |\eta(r)|,
$$
where $M=\|u_0\|_\infty$.

\end{corollary}
\begin{proof}
Assertion (i) readily follows from (\ref{est1}). Applying (i) to the entropies $\eta(u)=\pm u$, we obtain that for a.e. $t,\tau>0$
$$
\int_{\T^n}u(t,x)dx=\int_{\T^n}u(\tau,x)dx,
$$
and (\ref{mass}) follows in the limit as $\tau\to 0$ with the help of initial condition (iv) of Definition~\ref{def1}.

To prove (iii), we pass to the limits in (\ref{est1}) as $t_1\to 0$, $t_2\to\infty$, taking again into account the initial condition.
\end{proof}

\begin{lemma}\label{lemA2}
In the notations of Definition~\ref{def1} the distributions
$$\div_x B_r(u(t,x))\in L^2(\R_+\times\T^n) \quad \forall r=1,\ldots,n,$$ and for almost all $\tau>0$
\begin{eqnarray}\label{est2}
\int_{(\tau,+\infty)\times\T^n}\sum_{r=1}^n (\div_x B_r(u(t,x)))^2dtdx\le\nonumber\\ \frac{1}{2}\int_{\T^n}(u(\tau,x))^2dx\le \frac{1}{2}\int_{\T^n}(u_0(x))^2dx.
\end{eqnarray}
\end{lemma}

\begin{proof}
In view of relation (\ref{entr}) with $\eta(u)=u^2/2$
$$
\sum_{r=1}^n (\div_x B_r(u))^2\le - \{(u^2/2)_t+\div_x(T_{u}(\varphi)(u))-D^2\cdot T_{u}(A(u))\} \mbox{ in }
\D'(\Pi).
$$
Applying this relation to the nonnegative test function $k^{-n}\alpha(t)\beta(x/k)$, with $\alpha(t)\in C_0^1(\R_+)$, $\beta(y)\in C_0^2(\R^n)$, $k\in\N$, and passing to the limit as $k\to\infty$, we obtain, like in the proof of Lemma~\ref{lemA1}, that
$$
\int_{\R_+\times\T^n}\sum_{r=1}^n (\div_x B_r(u))^2\alpha(t)dtdx \le\frac{1}{2}\int_{\R_+\times\T^n} u^2\alpha'(t)dtdx.
$$
This relation implies that for almost all $\tau,T>0$, $\tau<T$
\begin{eqnarray*}
\int_{(\tau,T)\times\T^n}\sum_{r=1}^n (\div_x B_r(u))^2dtdx \le \\ \frac{1}{2}\int_{\T^n} (u(\tau,x))^2dx-\frac{1}{2}\int_{\T^n} (u(T,x))^2dx\le \frac{1}{2}\int_{\T^n} (u(\tau,x))^2dx,
\end{eqnarray*}
and (\ref{est2}) follows in the limit as $T\to+\infty$. We take also into account that for a.e. $t_0\in (0,\tau)$
$$
\int_{\T^n} (u(\tau,x))^2dx\le \int_{\T^n} (u(t_0,x))^2dx
$$
by Corollary~\ref{cor0}(i), which implies in the limit as $t_0\to 0$ that
$$
\int_{\T^n} (u(\tau,x))^2dx\le \int_{\T^n} (u_0(x))^2dx.
$$
\end{proof}

\begin{lemma}\label{lemA3}
Let $u_1=u_1(t,x)$, $u_2=u_2(t,x)$ be $x$-periodic e.s. of (\ref{1}), (\ref{2}) with initial functions $u_{10}(x), u_{20}(x)\in L^\infty(\T^n)$, respectively. Then for a.e. $t,\tau>0$ such that $t>\tau$
\begin{equation}\label{contr}
\int_{\T^n}|u_1(t,x)-u_2(t,x)|dx\le \int_{\T^n}|u_1(\tau,x)-u_2(\tau,x)|dx\le \int_{\T^n}|u_{10}(x)-u_{20}(x)|dx.
\end{equation}
\end{lemma}

\begin{proof}
As was established in \cite{ChPer1,BenKarl} (for the isotropic case see earlier paper \cite{Car}) by application of a variant of Kruzhkov's doubling variables method,
\begin{eqnarray*}
(|u_1-u_2|)_t+\div_x [\sgn(u_1-u_2)(\varphi(u_1)-\varphi(u_2))]- \\ D^2\cdot [\sgn(u_1-u_2)(A(u_1)-A(u_2))]\le 0 \ \mbox{ in }
\D'(\Pi).
\end{eqnarray*}
Applying this relation to a nonnegative test function $k^{-n}\alpha(t)\beta(x/k)$ like in the proofs of Lemmas~\ref{lemA1}-\ref{lemA2},  and passing to the limit as $k\to\infty$, we arrive at the relation
$$
\int_{\R_+\times\T^n}|u_1(t,x)-u_2(t,x)|\alpha'(t)dtdx\ge 0 \quad \forall \alpha(t)\in C_0^1(\R_+), \alpha(t)\ge 0.
$$
This means that
$$
\frac{d}{dt}\int_{\T^n}|u_1(t,x)-u_2(t,x)|dx\le 0 \ \mbox{ in } \D'(\R_+).
$$
Therefore, for a.e. $t,\tau,\delta\in\R_+$ such that $t>\tau>\delta$
$$
\int_{\T^n}|u_1(t,x)-u_2(t,x)|dx\le \int_{\T^n}|u_1(\tau,x)-u_2(\tau,x)|dx\le \int_{\T^n}|u_1(\delta,x)-u_2(\delta,x)|dx,
$$
and (\ref{contr}) follows in the limit as $\delta\to 0$.
\end{proof}

\begin{corollary}\label{cor1}
Let $u(t,x)$ be a periodic e.s. of (\ref{1}), (\ref{2}). Then the family $u(t,\cdot)$, $t>0$, is pre-compact in $L^2(\T^n)$ (after possible correction on a set of null measure of values $t$).
\end{corollary}

\begin{proof}
Let
$$
E=\{ \ t>0 \ | \ (t,x) \ \mbox{ is a Lebesgue point of } u \mbox{ for a.e. } x\in\R^n \ \}.
$$
It is known (see, for example, \cite[Lemma~1.2]{PaJHDE}) that $E$ is a set of full measure and $t\in E$ is a common Lebesgue point of the functions $t\to\int u(t,x)g(x)dx$ for all $g(x)\in L^1(\R^n)$. It is clear that for $t\in E$ and  $\Delta x\in\R^n$ $(t,x)$ is a Lebesgue point of $|u(t,x+\Delta x)-u(t,x)|$ for a.e. $x\in\R^n$. By \cite[Lemma~1.2]{PaJHDE} again, $t\in E$ is a Lebesgue point of the functions $\displaystyle t\to\int_{\T^n}|u(t,x+\Delta x)-u(t,x)|dx$ for all $\Delta x\in\R^n$.

Applying Lemma~\ref{lemA3} to e.s. $u(t,x)$, $u(t,x+\Delta x)$, we find that for each $t\in E$
$$
\int_{\T^n}|u(t,x+\Delta x)-u(t,x)|dx\le \int_{\T^n}|u_0(x+\Delta x)-u_0(x)|dx.
$$
and $|u(t,x)|\le M=\|u_0\|_\infty$. Therefore, for all $t\in E$ and all $\Delta x\in\R^n$
\begin{eqnarray*}
\int_{\T^n}(u(t,x+\Delta x)-u(t,x))^2dx\le 2M\int_{\T^n}|u(t,x+\Delta x)-u(t,x)|dx\le \\ 2M\int_{\T^n}|u_0(x+\Delta x)-u_0(x)|dx.
\end{eqnarray*}
This means that the family $u(t,\cdot)$, $t\in E$, is uniformly bounded and equicontinuous in $L^2(\T^n)$.
By the Kolmogorov-Riesz criterion this implies the pre-compactness of the family $u(t,\cdot)$ in $L^2(\T^n)$.
\end{proof}

\section{Reduction of the problem}\label{sec3}
We will need the following algebraic statement.

\begin{lemma}\label{lema}
Let $A$ be a lattice in $X=\R^n$, $A_0$ be a subgroup of $A$, $X_0$ be a linear span of $A_0$. Assume that $A_0=A\cap X_0$. Then any basis of $A_0$ can be completed to a basis of $A$.
\end{lemma}

\begin{proof}
Let $\xi_1,\ldots,\xi_k$ be a basis of $A_0$, that is, any element $\xi\in A_0$ is uniquely represented as
$\displaystyle\xi=\sum_{i=1}^k n_i\xi_i$ with integer coefficients $n_i$. It is clear that $A_0$ is a lattice and therefore $\xi_1,\ldots,\xi_k$ is a basis of the vector space $X_0$. We consider the natural projection ${\pr:X\to X/X_0}$.
Then $B=\pr(A)$ is an additive subgroup of $X/X_0$. We will show that $B$ is a lattice, i.e., a discrete subgroup of $X/X_0$.
It is sufficient to show that any ball $B_R=\{ \ x\in X/X_0 \ | \ p(x)\le R \ \}$ contains only finite set of points of $B$. Here
$$
p(x)=\min \{ \ |\xi-y| \ | \ y\in X_0 \ \}, \quad x=\pr(\xi),
$$
is the factor-norm. Here, and in the sequel, $|z|$ denotes the Euclidean norm of a finite-dimensional vector $z$. We assume that $x=\pr(\xi)\in B$, where $\xi\in A$, and that $p(x)\le R$.
We can choose $y\in X_0$ such that $|\xi-y|=p(x)\le R$. Recall that $\xi_i$, $i=1,\ldots,k$ is a basis of $X_0$. Therefore, we can represent $\displaystyle y=\sum_{i=1}^k s_i\xi_i$, $s_i\in\R$. Let $n_i=[s_i]\in \Z$ be integer parts of $s_i$,
so that $\alpha_i=s_i-n_i\in [0,1)$. Then
\begin{eqnarray*}
\left|\xi-\sum_{i=1}^k n_i\xi_i\right|=\left|\xi-y+\sum_{i=1}^k \alpha_i\xi_i\right|\le \\
|\xi-y|+\sum_{i=1}^k |\alpha_i||\xi_i|\le R+C, \quad C=\sum_{i=1}^k |\xi_i|=\const.
\end{eqnarray*}
Thus,
$$
\xi-\sum_{i=1}^k n_i\xi_i\in A_{R+C}=\{ \ \zeta\in A \ | \ |\zeta|\le R+C \ \}.
$$
Since $A$ is a lattice, the set $A_{R+C}$ is finite. Taking into account that $\displaystyle \sum_{i=1}^k n_i\xi_i\in A_0\subset X_0$, we find that
$$
x=\pr(\xi)=\pr\left(\xi-\sum_{i=1}^k n_i\xi_i\right)\in\pr(A_{R+C}).
$$
Hence, $B\cap B_R\subset\pr(A_{R+C})$ is a finite set for each $R>0$. We conclude that $B$ is a lattice. Therefore, there exists a basis $x_i$, $i=k+1,\ldots,m$, of the free abelian group $B$. We can find such $\xi_i\in A$ that $x_i=\pr(\xi_i)$.  Then for every $\xi\in A$ there exist unique
$n_i\in\Z$, $i=k+1,\ldots,m$, such that $\displaystyle\eta=\xi-\sum_{i=k+1}^m n_i\xi_i\in X_0$. We also observe that
$\eta\in A$. Thus, $\eta\in A\cap X_0=A_0$. Since $\xi_i$, $i=1,\ldots,k$ is a basis of $A_0$, then there exist unique $n_i\in\Z$, $i=1,\ldots,k$, such that
$\displaystyle
\eta=\sum_{i=1}^k n_i\xi_i.$ We conclude that
$$
\xi=\sum_{i=1}^m n_i\xi_i
$$
and this representation is unique. This means that $\xi_i$, $i=1,\ldots, m$ is a basis of $A$. The proof is complete.
\end{proof}

\begin{remark}\label{rem1}
The requirement $A\cap X_0=A_0$ is exact: if a basis of $A_0$ can be completed to a basis of $A$ then $A\cap X_0=A_0$. In fact, let $\xi_1,\ldots,\xi_k,\xi_{k+1},\ldots,\xi_m$ be a basis of $A$ completing a basis $\xi_1,\ldots,\xi_k$ of $A_0$ and $\xi\in A\cap X_0$. Since $\xi_1,\ldots,\xi_k$ is also a basis of the linear space $X_0$,  we have two representation
$$
\xi=\sum_{i=1}^m n_i\xi_i=\sum_{i=1}^k s_i\xi_i, \quad n_i\in\Z, \ s_i\in\R.
$$
Taking into account that the vectors $\xi_1,\ldots,\xi_m$ are linearly independent, we conclude that these representations must coincide. In particular, $n_i=0$ for $i>k$ and $\displaystyle\xi=\sum_{i=1}^k n_i\xi_i\in A_0$. Hence, $A\cap X_0\subset A_0$. Since the inverse inclusion $A_0\subset A\cap X_0$ is evident, we conclude that $A\cap X_0=A_0$.
\end{remark}

\medskip
We define
$$
L'_0=\{ \ \xi\in L' \ | \ \mbox{ the function } \xi\cdot\varphi(u) \ \mbox{ is affine in some vicinity of } I \ \}.
$$
It is clear that for any vector $\xi\in X_0$, where $X_0$ is the linear span of $L'_0$, the function $\xi\cdot\varphi(u)$ is affine in some vicinity of $I$, this means that $\xi\cdot\varphi(u)=a(\xi)u+\const$ in this vicinity. Since the map $\xi\to a(\xi)$ is linear, there exist a unique vector $\bar c\in X_0$ such that $a(\xi)=\bar c\cdot\xi$.

Obviously, $L'\cap X_0\subset L'_0$ and $L'_0$ is a subgroup of the lattice $L'$ satisfying the condition of Lemma~\ref{lema}.
Let $m$ be a rank of $L'_0$, $d=n-m$, $\zeta_i$, $i=d+1,\ldots,n$, be a basis of $L'_0$. By Lemma~\ref{lema} this basis can be completed to a basis $\zeta_i$,
$i=1,\ldots,n$, of the lattice $L'$. Then the vectors $\zeta_i$, $i=1,\ldots,n$ forms a basis of $\R^n$ as well.
We introduce the linear change $y=y(t,x)$, $y_i=\zeta_i\cdot x-c_it$, $i=1,\ldots,n$, where $c_i=\bar c\cdot\zeta_i=a(\zeta_i)$. After this change, we obtain entropy solution $v(t,y)$ of the problem
$$
v_t+\div_y(\tilde\varphi(v)-\tilde a(v)\nabla_y v)=0, \quad v(0,y)=v_0(y)=u_0(x(0,y)),
$$
where, in view of (\ref{after}), $\tilde\varphi_i(v)=\zeta_i\cdot\varphi(v)-c_iv$, $\tilde a_{ij}(v)=a(v)\zeta_i\cdot\zeta_j$. By the construction
the last $m$ flux component $\tilde\varphi_i(v)$, $i=d+1,\ldots,n$ are constant on some interval $(\alpha,\beta)\ni I$.
We notice that $v(t,y)$ is periodic with respect to $y$ with the standard lattice of periods $\Z^n$, this follows from the fact that $y(0,x)$ is an isomorphism from $L$ into $\Z^n$. We underline that the mean value $I$ remains the same under the described change. The nonlinearity-diffusivity condition of Theorem~\ref{thM} converts to the requirement: for all $\kappa\in\Z^n$, $\kappa\not=0$, there is no vicinity of $I$, where simultaneously the function $\kappa\cdot\tilde\varphi(u)$ is affine and the function $\tilde a(u)\kappa\cdot\kappa\equiv 0$. This directly follows from the relations
$$
\kappa\cdot\tilde\varphi(u)=\xi\cdot\varphi(u), \quad \tilde a(u)\kappa\cdot\kappa=a(u)\xi\cdot\xi,
$$
where $\displaystyle\xi=\sum_{i=1}^n\kappa_i\zeta_i\in L'$. The condition that the function $\xi\cdot\varphi(u)$ is affine in some vicinity of $I$ is equivalent to the inclusion $\xi\in L'_0 \Leftrightarrow \tilde\kappa=0$, where we denote
by $\tilde\kappa$, $\bar\kappa$ the orthogonal projections of $\kappa$ on the spaces
\begin{eqnarray*}\R^d=\{ \ \xi=(\xi_1,\ldots,\xi_n)\in\R^n \ | \ \xi_i=0, \ \forall i=d+1,\dots,n \ \}, \\ (\R^d)^\perp=\{\ \xi=(\xi_1,\ldots,\xi_n)\in\R^n \ | \ \xi_i=0, \ \forall i=1,\dots,d \ \},
\end{eqnarray*}
respectively. Therefore, the nonlinearity-diffusivity condition reduces to the requirement:
$\forall\kappa=\bar\kappa\in\Z^n$, $\kappa\not=0$, the function $\tilde a(u)\kappa\cdot\kappa\not\equiv 0$ in any vicinity of $I$.

Hence, going back to the original notations of equation (\ref{1}), we can suppose that
\begin{itemize}

\item[(R1)] $L=\Z^n$ (and in particular the volume $V=1$);

\item[(R2)] the function $\xi\cdot\varphi(u)$, $\xi\in L'=\Z^n$, is affine if and only if $\xi\in (\R^d)^\perp$;

\item[(R3)] the function $\xi\cdot\varphi(u)\equiv c(\xi)=\const$ on an interval $(\alpha,\beta)\ni I$ for all $\xi\in \Z^n\cap (\R^d)^\perp$;

\item[(R4)] the function $a(u)\xi\cdot\xi\not\equiv 0$ in any vicinity of $I$ for all $\xi\in \Z^n\cap (\R^d)^\perp$.
\end{itemize}

\section{Preliminaries}

\subsection{Measure valued functions and H-measures}
Recall (see \cite{Di,Ta1}) that a measure-valued function on a domain $\Omega\subset\R^N$ is a weakly
measurable map $x\mapsto \nu_x$ of $\Omega$ into the space $\operatorname{Prob}_0(\R)$ of probability
Borel measures with compact support in~$\R$.

The weak measurability of $\nu_x$ means that for each continuous function $g(\lambda)$ the
function $x\to\langle\nu_x,g(\lambda)\rangle\doteq\int g(\lambda)d\nu_x(\lambda)$ is measurable on~$\Omega$.

A  measure-valued function $\nu_x$  is said to be bounded if there
exists $M>0$ such  that $\supp\nu_x\subset[-M,M]$  for almost all
$x\in\Omega$.

Measure-valued functions of the kind
$\nu_x(\lambda)=\delta(\lambda-u(x))$, where $u(x)\in L^\infty(\Omega)$ and
$\delta(\lambda-u^*)$ is the Dirac measure at $u^*\in\R$, are called {\it regular}. We identify these
measure-valued functions and the corresponding functions $u(x)$, so that there is a natural
embedding of $L^\infty(\Omega)$ into the set $\MV(\Omega)$ of bounded measure-valued functions on~$\Omega$.

Measure-valued functions naturally arise as weak limits of bounded sequences in $L^\infty(\Omega)$ in the
sense of the following theorem by L.~Tartar \cite{Ta1}.

\begin{theorem}\label{thT}
Let $u_k(x)\in L^\infty(\Omega)$, $k\in\N$, be a bounded sequence. Then there exist a subsequence (we
keep the notation $u_k(x)$ for this subsequence) and a bounded measure valued function $\nu_x\in\MV(\Omega)$
such that
\begin{equation} \label{pr2} \forall g(\lambda)\in C(\R) \quad g(u_k)
\mathop{\to}_{k\to\infty}\langle\nu_x,g(\lambda)\rangle \quad\text{weakly-\/$*$ in } L^\infty(\Omega).
\end{equation}
Besides, $\nu_x$ is regular, i.e., $\nu_x(\lambda)=\delta(\lambda-u(x))$ if and only if $u_k(x)
\mathop{\to}\limits_{k\to\infty} u(x)$ in $L^1_{loc}(\Omega)$ (strongly).
\end{theorem}

Another useful tool for evaluation of weak convergence is Tartar's H-measures and their variants. H-measures were introduced by L. Tartar \cite{Ta2} and independently by P.~G\'erard in \cite{Ger}.
We recall the notion of H-measure in the simple case of scalar sequences.

Let $S=\{ \ \xi\in\R^N \ | \ |\xi|=1 \ \}$ be the unit sphere in $\R^N$,
$\pi:\R^N\setminus\{0\}\to S$, $\pi(\xi)=\xi/|\xi|$, be the orthogonal projection on the sphere.
Let $$F(u)(\xi)=\int_{\R^N} e^{-2\pi i\xi\cdot x} u(x)dx, \quad \xi\in\R^N,$$
be the Fourier transformation extended as a unitary operator on the space $L^2(\R^N)$, $u\to\overline{u}$, $u\in\C$, be the complex conjugation.

Now, we assume that $u_k=u_k(x)$ is a bounded sequence in $L^\infty(\Omega)$ (more generally, in $L^2_{loc}(\Omega)$) weakly convergent to $0$,

\begin{proposition}[see \cite{Ta2,Ger}]\label{pro1} There exists a nonnegative Borel measure
$\mu$ in $\Omega\times  S$ and a subsequence $u_r(x)=u_k(x),$ $k=k_r,$ such that
\begin{equation}\label{Hm}
\<\mu,\Phi_1(x)\overline{\Phi_2(x)}\psi(\xi)\>=
\lim\limits_{r\to\infty}\int_{\R^n}
F(u_r\Phi_1)(\xi)\overline{F(u_r\Phi_2)(\xi)}
\psi(\pi(\xi))d\xi
\end{equation}
for all $\Phi_1(x),\Phi_2(x)\in C_0(\Omega)$ and $\psi(\xi)\in
C(S)$.
\end{proposition}

\begin{remark}\label{rem2}
It follows from (\ref{Hm}) and the Plancherel identity that $\pr_\Omega\mu\le C\meas$, and that
(\ref{Hm}) remains valid for all $\Phi_1(x),\Phi_2(x)\in L^2(\Omega)$, cf. \cite[Remark 2(a)]{PaARMA}. Here
we denote by $\meas$ the Lebesgue measure on $\Omega$.
\end{remark}

The measure $\mu$ is called an H-measure corresponding to $u_r(x)$.
It is clear that $\mu=0$ if and only if the sequence $u_r\to 0$ as $r\to\infty$ in $L^2_{loc}(\Omega)$. In some sense, the H-measure $\mu=\mu(x,\xi)$ indicates the strength of oscillations of the sequence $u_r$ at the point $x$ and in the direction $\xi$.

In the sequel we will use the special sequences obtained from a given function by rescaling of independent variables.
Namely, let $u(t,x)\in L^\infty(\Pi)$ be a function periodic over the space variables $x$ (with the standard lattice of periods $\Z^n$), and $s_k\in\R_+$, $p_k,q_k\in\N$ be sequences converging to infinity as $k\to\infty$. We consider the sequence $u_k=u_k(t,x)=u(s_kt,p_k\tilde x+q_k\bar x)$, where $\tilde x,\bar x$ are the orthogonal projection of $x\in\R^n$ into the spaces $\R^d=\{\ x=(x_1,\ldots,x_n) \ | \ x_i=0, i=d+1,\ldots, n \ \}$ and $(\R^d)^\perp=\{\ x=(x_1,\ldots,x_n) \ | \ x_i=0, i=1,\ldots,d \ \}$, respectively. By the periodicity, the function $u$ and the functions $u_k$ can be treated as elements of $L^\infty(\R_+\times\T^n)$, where $\T^n=\R^n/\Z^n$ is the standard torus.

\begin{lemma}\label{lem1}
Suppose that the sequence $I_0(s_kt)=\int_{\T^n} u(s_kt,x)dx\rightharpoonup v(t)$ weakly-$*$ in $L^\infty(\R_+)$ as $k\to\infty$. Then
$u_k\rightharpoonup v(t)$ weakly-$*$ in $L^\infty(\Pi)$ as $k\to\infty$.
\end{lemma}

\begin{proof}
For given $r\in\Z^n$, $r\not=0$, we choose $\bar k\in\N$ so large that $\min(p_k,q_k)>|r|^2$ for all $k\ge\bar k$. Observe that the projections $\tilde r, \bar r\in\Z^n$ and, by the periodicity of $u(t,\cdot)$,
\begin{eqnarray}\label{3}
\int_{\T^n}e^{2\pi ir\cdot x} u_k(t,x)dx=\int_{\T^n}e^{2\pi ir\cdot (x+\tilde r/p_k)} u(s_kt,p_k\tilde x+\tilde r+q_k\bar x)dx=\nonumber\\ \int_{\T^n}e^{2\pi ir\cdot (x+\tilde r/p_k)} u(s_kt,p_k\tilde x+q_k\bar x)dx=e^{2\pi i|\tilde r|^2/p_k}\int_{\T^n}e^{2\pi ir\cdot x} u_k(t,x)dx, \\
\label{4}
\int_{\T^n}e^{2\pi ir\cdot x} u_k(t,x)dx=\int_{\T^n}e^{2\pi ir\cdot (x+\bar r/q_k)} u(s_kt,p_k\tilde x+q_k\bar x+\bar r)dx=\nonumber\\ \int_{\T^n}e^{2\pi ir\cdot (x+\bar r/q_k)} u(s_kt,p_k\tilde x+q_k\bar x)dx=e^{2\pi i|\bar r|^2/q_k}\int_{\T^n}e^{2\pi ir\cdot x} u_k(t,x)dx.
\end{eqnarray}
Since either $\tilde r\not=0$ or $\bar r\not=0$ while $p_k>|r|^2\ge |\tilde r|^2$, $q_k>|r|^2\ge |\bar r|^2$, one of the factors $e^{2\pi i|\tilde r|^2/p_k}$, $e^{2\pi i|\bar r|^2/q_k}$ is different from $1$. Therefore, it follows from (\ref{3}), (\ref{4}) that for $k\ge\bar k$
$$
\int_{\T^n}e^{2\pi ir\cdot x} u_k(t,x)dx=0 \quad \forall t>0.
$$
This implies that for each $a(t)\in L^1(\R_+)$
\begin{equation}\label{5}
\int_{\R^+\times\T^n} a(t)e^{2\pi ir\cdot x} u_k(t,x)dtdx\mathop{\to}_{k\to\infty} 0=\int_{\R^+\times\T^n} v(t)a(t)e^{2\pi ir\cdot x}dtdx.
\end{equation}
If $r=0$ we have
$$
\int_{\T^n} u_k(t,x)dx=\int_{\T^n} u(s_kt,p_k\tilde x+q_k\bar x)dx=\int_{\T^n} u(s_kt,x)dx=I_0(s_kt)
$$
because the map $x\to p_k\tilde x+q_k\bar x$ keeps the Lebesgue measure on $\T^n$.
Since $I_0(s_kt)$ weakly-$*$ converges to the function $v(t)$, we conclude that
for each $a(t)\in L^1(\R_+)$
\begin{equation}\label{6}
\int_{\R^+\times\T^n} a(t) u_k(t,x)dtdx\mathop{\to}_{k\to\infty} \int_{\R_+} v(t)a(t)dt=\int_{\R^+\times\T^n} v(t)a(t)dtdx.
\end{equation}
From relations (\ref{5}), (\ref{6}) it follows that
\begin{equation}\label{7}
\int_{\R^+\times\T^n} f(t,x) u_k(t,x)dtdx\mathop{\to}_{k\to\infty} \int_{\R^+\times\T^n} v(t)f(t,x)dtdx
\end{equation}
for every function $f(t,x)=\sum_{r\in I} a_r(t)e^{2\pi ir\cdot x}$, where the parameter $r$ runs over a finite set $I\subset\Z^n$ while $a_r(t)\in L^1(\R_+)$. Since the space of such functions is dense in $L^1(\R_+\times\T^n)$, relation (\ref{7}) implies that
$u_k\rightharpoonup v(t)$ as $k\to\infty$ weakly-$*$ in $L^\infty(\R_+\times\T^n)$. By the periodicity, the above limit relation holds also  weakly-$*$ in the space $L^\infty(\Pi)$. The proof is complete.
\end{proof}

\begin{corollary}\label{cor2}
Assume in addition that for a dense set of functions $p(u)\in C(\R)$
\begin{equation}\label{8}
\int_{\T^n}p(u(s_kt,x))dx\mathop{\rightharpoonup}_{k\to\infty} c_p=\const \mbox{ weakly-$*$ in } L^\infty(\R_+).
\end{equation}
Then the sequence $u_k$ converges to a constant measure valued function ${\nu_{t,x}\equiv\nu}$ in the sense of relation (\ref{pr2}).
\end{corollary}

\begin{proof}
Since the set of functions $p(\lambda)$, for which relation (\ref{8}) holds, is dense in $C([-M,M])$, where $M=\|u\|_\infty$, this relation remains valid for all $p(\lambda)\in C([-M,M])$. It is clear that the functional $p\to c_p$ is linear and continuous. By the Riesz-Markov representation theorem $c_p=\langle\nu,p(\lambda)\rangle=\int p(\lambda)d\nu(\lambda)$ with some Borel measure $\nu$ on $[-M,M]$. Evidently, $c_p\ge 0$ for $p(\lambda)\ge 0$ and $c_p=1$ for $p\equiv 1$, which implies that $\nu$ is a probability measure. By Lemma~\ref{lem1} we conclude that for each $p(\lambda)\in C(\R)$
$$
p(u_k)\mathop{\rightharpoonup}_{k\to\infty} c_p=\langle\nu,p(\lambda)\rangle \mbox{ weakly-$*$ in } L^\infty(\Pi),
$$
which is exactly (\ref{pr2}) with $\nu_{t,x}\equiv \nu$.
\end{proof}

\begin{remark}\label{rem3}
Condition (\ref{8}) is always satisfied in the case when $u(t,x)$ is a periodic e.s. of (\ref{1}), (\ref{2}). In fact,
any function $p(u)\in C^2(\R)$ is a difference of two convex functions. By Corollary~\ref{cor0}(i) the function
$\displaystyle I_p(t)=\int_{\T^n}p(u(t,x))dx$ is a difference of two decreasing functions (after possible extraction of a set of null measure).
Hence $I_p(t)$ is a function of bounded variation and there exists a limit $c_p$ of this function as $t\to+\infty$. This implies that the sequence $I_p(s_kt)$ converges in $L^1_{loc}(\R_+)$ to the constant $c_p$. Hence, limit relation (\ref{8}) holds, even in the stronger topology of $L^1_{loc}$.
\end{remark}

Let $u_k=u(k^2t,k^2\tilde x+k\bar x)$ be a sequence of the considered above kind (with $s_k=p_k=k^2$, $q_k=k$).
Passing to a subsequence, we may assume that this sequence converges weakly-$*$ in $L^\infty(\Pi)$. Then the sequence
$$I_0(k^2t)=\int_{\T^n} u(k^2t,x)dx=\int_{\T^n} u_k(t,x)dx$$ converges weakly-$*$ in $L^\infty(\R_+)$ to some function $v(t)$. By Lemma~\ref{lem1} the function $v(t)$ is a weak-$*$ limit of the sequence $u_k$.

We are going to investigate the H-measure corresponding to the sequence $u_k-v(t)$.
We suppose that the family $u(t,\cdot)$, $t>0$, is pre-compact in $L^2(\T^n)$. As was shown in \cite[Lemma~3.1]{PaAIHP}, this requirement implies that the Fourier series
\begin{equation}\label{fourier}
u(t,x)=\sum_{\kappa\in\Z^n} a_\kappa(t)e^{2\pi i \kappa\cdot x}, \quad a_\kappa(t)=\int_{\T^n}e^{-2\pi i \kappa\cdot x}u(t,x)dx,
\end{equation}
converge in $L^2(\T^n)$ uniformly in $t\ge 0$.

By Proposition~\ref{pro1} there exists $\mu=\mu(t,x,\tau,\xi)$, an H-measure corresponding to some subsequence
of $u_k-v$. Denote by
$$
S_0=\left\{ \pi(\hat\xi)\in S  \ | \ \hat\xi=(\tau,\xi)\not=0, \ \tau\in\R,
\xi\in\Z^n, |\tilde\xi|\cdot|\bar\xi|=0 \ \right\},
$$
where
$$
\pi(\tau,\xi)=\frac{(\tau,\tilde\xi)}{(\tau^2+|\xi|^2)^{1/2}}
$$
is the orthogonal projection on the sphere
$$
S=\{ \ (\tau,\xi)\in\R^{n+1} \ | \ \tau^2+|\xi|^2=1 \ \}.
$$

\begin{proposition}\label{pro2}
The support of H-measure $\mu$ is contained in $\Pi\times S_0$, that is, $\mu(\Pi\times(S\setminus S_0))=0$.
\end{proposition}

\begin{proof}
For $m\in\N$ we introduce the sets
$$S_m= \left\{ \pi(\hat\xi)\in S \ | \ \hat\xi=(\tau,\xi)\not=0, \ \tau\in\R,
\xi\in\Z^n, |\tilde\xi|\cdot|\bar\xi|=0, |\xi|\le m \ \right\}.$$ It is clear that $S_m$ is a closed subset of $S$ (it
is the union of the finite set of circles
$\{ \ \bigl(p,q|\xi|^{-1}\xi\bigr) \ | \ p^2+q^2=1 \ \}$, where
$\xi\in\Z^n$, $|\tilde\xi|\cdot|\bar\xi|=0$, $0<|\xi|\le m$), and $\displaystyle S_0=\cup_{m=1}^\infty S_m$. Let
$$
u(t,x)=\sum_{\kappa\in\Z^n} a_\kappa(t)e^{2\pi i\kappa\cdot x}
$$
be the Fourier series for $u(t,\cdot)$ in $L^2(\T^n)$. Then
\begin{equation}\label{ser}
u_k(t,x)=u(k^2t,k^2\tilde x+k\bar x)=\sum_{\kappa\in\Z^n} a_\kappa(k^2t)e^{2\pi i (k^2\tilde\kappa+k\bar\kappa)\cdot x}.
\end{equation}
We denote $b_{0,k}=a_0(k^2t)-v(t)$; $b_{\kappa, k}=a_\kappa(k^2t)$, where $\kappa\in\Z^n$,
$\kappa\not=0$. Let $\alpha(t)\in C_0(\R_+)$, and $\beta(x)\in L^2(\R^n)\cap C^\infty(\R^n)$ be such function
that its Fourier transform is a continuous compactly supported function:
\begin{equation}\label{reg}
\tilde\beta(\xi)=\int_{\R^n}e^{-2\pi i\xi\cdot x}\beta(x)dx\in C_0(\R^n).
\end{equation}
We take $R=\max\limits_{\xi\in\supp\tilde\beta} |\xi|$. Let $\Phi(t,x)=\alpha(t)\beta(x)$.
By (\ref{ser}) we find that
\begin{equation}\label{ser1}
(u_k(t,x)-v(t))\Phi(t,x)=\sum_{\kappa\in\Z^n} b_{\kappa,
k}(t)\alpha(t)e^{2\pi i (k^2\tilde\kappa+k\bar\kappa)\cdot x}\beta(x).
\end{equation}
Observe that the Fourier transform of $e^{2\pi i (k^2\kappa+k\bar\kappa)\cdot x}\beta(x)$ in $\R^n$ coincides with
${\tilde\beta(\xi-(k^2\tilde\kappa+k\bar\kappa))}$. Since for $k>2R+1$ supports of these functions do not intersect, then
for such $k$ the series
\begin{equation}\label{ser2}
\sum_{\kappa\in\Z^n} b_{\kappa,
k}(t)\alpha(t)\tilde\beta(\xi-(k^2\tilde\kappa+k\bar\kappa))
\end{equation}
is orthogonal in $L^2(\R^n)$ for each $t>0$. Besides, by the Plancherel equality
$$\|\tilde\beta(\xi-(k^2\tilde\kappa+k\bar\kappa))\|_{L^2(\R^n)}=\|\tilde\beta\|_2=\|\beta\|_2,$$ and
\begin{eqnarray*}
\sum_{\kappa\in\Z^n} |b_{\kappa,
k}(t)\alpha(t)|^2\|\tilde\beta(\xi-(k^2\tilde\kappa+k\bar\kappa))\|_{L^2(\R^n)}^2=\\
|\alpha(t)|^2\|\beta\|_2^2\sum_{\kappa\in\Z^n} |b_{\kappa,k}(t)|^2=
|\alpha(t)|^2\|\beta\|_2^2\cdot\|u(k^2t,\cdot)-v(t)\|_{L^2(\T^n)}^2<+\infty.
\end{eqnarray*}
Therefore, orthogonal series (\ref{ser2}) converges in $L^2(\R^n)$ for each $t>0$. Moreover, by the uniform converges of the Fourier series (\ref{fourier})
\begin{equation}\label{9}
\sum_{\kappa\in\Z^n,|\kappa|>N}|b_{\kappa,k}(t)|^2=
\sum_{\kappa\in\Z^n,|\kappa|>N}|a_\kappa(k^2t)|^2\mathop{\to}_{N\to\infty} 0
\end{equation}
uniformly with respect to $t>0$. Hence, series (\ref{ser2}) converges in $L^2(\R^n)$ uniformly with
respect to $t$. Since the Fourier transformation is an isomorphism on $L^2(\R^n)$, we conclude that
series (\ref{ser1}) also converges in $L^2(\R^n)$ uniformly with respect to
$t$. Since $\alpha(t)\in C_0(\R)$, this implies that (\ref{ser1}) converges in $L^2(\Pi)$, and
\begin{equation}\label{l0}
F((u_k-v)\Phi)(\hat\xi)=\sum_{\kappa\in\Z^n}
F^t(\alpha b_{\kappa,k})(\tau)\tilde\beta(\xi-(k^2\tilde\kappa+k\bar\kappa)), \ \hat\xi=(\tau,\xi),
\end{equation}
where $F^t(h)(\tau)=\int_{\R}e^{-2\pi i\tau t} h(t)dt$ denotes the Fourier transform over the time
variable (we extend functions $h(t)\in L^2(\R_+)$ on the whole line $\R$, setting $h(t)=0$ for $t<0$).
It follows from (\ref{l0}) that for $k>2R+1$
\begin{eqnarray}\label{l1}
\int_{\R^{n+1}} |F(\Phi(u_k-v))(\hat\xi)|^2\psi(\pi(\hat\xi))d\hat\xi=\nonumber\\
\sum_{\kappa\in\Z^n}\int_{\R^{n+1}}
|F^t(\alpha b_{\kappa,k})(\tau)|^2|\tilde\beta(\xi-(k^2\tilde\kappa+k\bar\kappa))|^2\psi(\pi(\hat\xi))d\hat\xi,
\end{eqnarray}
where $\psi(\hat\xi)\in C(S)$ is arbitrary. Now we fix $\varepsilon>0$. Recall that
$b_{\kappa,k}=a_\kappa(k^2t)$ for $\kappa\not=0$, and by (\ref{9}) there exists $m\in\N$ such
that
$$
\sup_{t>0}\sum_{\kappa\in\Z^n,|\kappa|>m} |a_\kappa(t)|^2<\varepsilon.
$$
Then
\begin{eqnarray}\label{l2}
\sum_{\kappa\in\Z^n,|\kappa|>m}\int_{\R^{n+1}}
|F^t(\alpha b_{\kappa,k})(\tau)|^2|\tilde\beta(\xi-(k^2\tilde\kappa+k\bar\kappa))|^2d\hat\xi=
\nonumber\\ \sum_{\kappa\in\Z^n,|\kappa|>m}\int_{\Pi}|\alpha(t) a_\kappa(k^2t)|^2|\beta(x)|^2dtdx\le
\nonumber\\  \|\Phi\|_2^2\cdot\sup_{t>0}\sum_{\kappa\in\Z^n,|\kappa|>m} |a_\kappa(t)|^2<\varepsilon\|\Phi\|_2^2.
\end{eqnarray}
Now we suppose that $\|\psi\|_\infty\le 1$ and $\psi(\hat\xi)=0$ on the set $S_m$. By (\ref{l2})
\begin{eqnarray}\label{l2a}
\sum_{\kappa\in\Z^n,|\kappa|>m}\int_{\R^{n+1}}
|F^t(\alpha b_{\kappa,k})(\tau)|^2|\tilde\beta(\xi-(k^2\tilde\kappa+k\bar\kappa))|^2|\psi(\pi(\hat\xi))|d\hat\xi<\varepsilon\|\Phi\|_2^2.
\end{eqnarray}
Since continuous function $\psi(\hat\xi)$ is uniformly continuous on the compact $S$ then we can find
such $\delta>0$ that $|\psi(\hat\xi_1)-\psi(\hat\xi_2)|<\varepsilon$ whenever $\hat\xi_1,\hat\xi_2\in
S$, $|\hat\xi_1-\hat\xi_2|<\delta$.
Suppose that $\kappa\not=0$, $\tilde\beta(\xi-(k^2\tilde\kappa+k\bar\kappa))\not=0$.
Then $|\xi-(k^2\tilde\kappa+k\bar\kappa)|\le R$. For a fixed $\tau\in\R$ we denote $\hat\xi=(\tau,\xi)$,
$\hat\eta=(\tau,k^2\tilde\kappa+k\bar\kappa)$; $\hat\eta_1=(\tau,k^2\tilde\kappa)$ if $\tilde\kappa\not=0$, $\hat\eta_1=(\tau,k\bar\kappa)$ if $\tilde\kappa=0$. By the evident estimate $\displaystyle|\pi(x)-\pi(y)|\le \frac{2|x-y|}{|y|}$, for each $\kappa\in\Z^n$, $\kappa\not=0$,
\begin{eqnarray}\label{l3a}
\left|\pi(\hat\xi)-\pi(\hat\eta)\right|\le\frac{2|\xi-(k^2\tilde\kappa+k\bar\kappa)|}{|\hat\eta|}\le \frac{2R}{k|\kappa|}\le\frac{2R}{k}, \\
\label{l3b}
\left|\pi(\hat\eta)-\pi(\hat\eta_1)\right|\le\frac{2k|\bar\kappa|}{k^2|\tilde\kappa|}\le
\frac{2|\bar\kappa|}{k} \ \mbox{ if } \tilde\kappa\not=0; \quad
\hat\eta=\hat\eta_1 \ \mbox{ if } \tilde\kappa=0.
\end{eqnarray}
In view of (\ref{l3a}), (\ref{l3b}), we have
\begin{equation}\label{l3}
\left|\pi(\hat\xi)-\pi(\hat\eta_1)\right|\le\frac{2(R+|\bar\kappa|)}{k}
\end{equation}
By (\ref{l3})
for all $k>k_0=2(R+m)/\delta$ and all $\kappa\in\Z^n$ such that $0<|\kappa|\le m$ and $\tilde\beta(\xi-(k^2\tilde\kappa+k\bar\kappa))\not=0$ the inequality $|\pi(\hat\xi)-\pi(\hat\eta_1)|<\delta$ holds, which implies the estimate
\begin{equation}\label{l4}
|\psi(\pi(\hat\xi))|=|\psi(\pi(\hat\xi))-\psi(\pi(\hat\eta_1)|<\varepsilon.
\end{equation}
We use here that $\pi(\hat\eta_1)\in S_m$ and, therefore, $\psi(\pi(\hat\eta_1))=0$. In view
of~(\ref{l4}), for all $k>k_0$
\begin{eqnarray}\label{l5}
\sum_{\kappa\in\Z^n, 0<|\kappa|\le m}\int_{\R^{n+1}} |F^t(\alpha
b_{\kappa,k})(\tau)|^2|\tilde\beta(\xi-(k^2\tilde\kappa+k\bar\kappa))|^2|\psi(\pi(\hat\xi))|d\hat\xi\nonumber\\
\le\varepsilon\sum_{\kappa\in\Z^n, 0<|\kappa|\le m}\int_{\R^{n+1}} |F^t(\alpha
b_{\kappa,k})(\tau)|^2|\tilde\beta(\xi-(k^2\tilde\kappa+k\bar\kappa))|^2d\hat\xi\le\nonumber\\
\varepsilon\|\beta\|_2^2\sum_{\kappa\in\Z^n}
\int_\R|\alpha(t) b_{\kappa,k}(t)|^2dt\le
\varepsilon\|\Phi\|_2^2\sup_{t>0}\sum_{\kappa\in\Z^n}|b_{\kappa,k}(t)|^2=\nonumber\\
\varepsilon\|\Phi\|_2^2\sup_{t>0}
\|u(k^2t,\cdot)-v(t)\|_{L^2(\T^n)}^2\le C\varepsilon\|\Phi\|_2^2,
\end{eqnarray}
where $C=4\|u\|_\infty^2$. Further, in the case $\kappa=0$ when $\hat\eta_1=\hat\eta=(\tau,0)\in S_m$, we have the estimate
$$
\left|\pi(\hat\xi)-\pi(\hat\eta)\right|\le\frac{2|\xi|}{|\tau|}\le \frac{2R}{|\tau|}
$$
for
$|\xi|\le R$. Taking $|\tau|>T\doteq\frac{2R}{\delta}$, we find
$$
|\psi(\pi(\hat\xi))|=|\psi(\pi(\hat\xi))-\psi(\pi(\hat\eta))|<\varepsilon.
$$
Therefore,
\begin{eqnarray}\label{l10}
\int_{\R^{n+1}}\theta(|\tau|-T)|F^t(\alpha b_{0,k})(\tau)|^2|\tilde\beta(\xi)|^2
|\psi(\pi(\hat\xi))|d\hat\xi\le C\varepsilon\|\Phi\|_2^2.
\end{eqnarray}
Here $\theta(r)=\left\{\begin{array}{lcr} 1 & , & r>0, \\ 0 & , & r\le 0 \end{array}\right.$
is the Heaviside function.

For $|\tau|\le T$ we are reasoning in the following way. Since
$\alpha(t)b_{0,k}(t)={\alpha(t)(a_{0,k}(t)-v(t))\rightharpoonup 0}$ as $k\to\infty$, and $\|\alpha
b_{0,k}\|_1\le C_1=2\|u\|_\infty\|\alpha\|_1$, the Fourier transform $F^t(\alpha
b_{0,k})(\tau)\mathop{\to}\limits_{k\to\infty} 0$ for all $\tau\in\R$ and uniformly bounded:
$|F^t(\alpha b_{0,k})(\tau)|\le C_1$. By Lebesgue dominated convergence theorem
$$
\int_{\R}\theta(T-|\tau|)|F^t(\alpha b_{0,k})(\tau)|^2 d\tau\mathop{\to}_{k\to\infty} 0.
$$
Therefore (recall that $\|\psi\|_\infty\le 1$),
\begin{eqnarray}\label{l11}
\int_{\R^{n+1}}\theta(T-|\tau|)|F^t(\alpha b_{0,k})(\tau)|^2|\tilde\beta(\xi)|^2
|\psi(\pi(\hat\xi))|d\hat\xi\le \nonumber\\ \|\beta\|_2\int_{\R}\theta(T-|\tau|)|F^t(\alpha
b_{0,k})(\tau)|^2 d\tau\mathop{\to}_{k\to\infty} 0.
\end{eqnarray}
In view of (\ref{l10}), (\ref{l11}) we find
\begin{equation}\label{l12}
\limsup_{k\to\infty}\int_{\R^{n+1}}
|F^t(\alpha b_{0,k})(\tau)|^2|\tilde\beta(\xi)|^2|\psi(\pi(\hat\xi))|
d\hat\xi\le C\varepsilon\|\Phi\|_2^2.
\end{equation}
Using (\ref{l1}), (\ref{l2a}), (\ref{l5}) and (\ref{l12}), we arrive at the relation
\begin{equation}\label{l13}
\limsup_{k\to\infty}\int_{\R^{n+1}} |F(\Phi(u_k-v))(\hat\xi)|^2|\psi(\pi(\hat\xi))|d\hat\xi\le
C_2\varepsilon\|\Phi\|_2^2,
\end{equation}
where $C_2$ is a constant. By the definition of H-measure and
Remark~\ref{rem2}
\begin{eqnarray*}
\lim_{k\to\infty}\int_{\R^{n+1}} |F(\Phi(u_k-v))(\hat\xi)|^2|\psi(\pi(\hat\xi))|d\hat\xi=\\
\<\mu,|\Phi(t,x)|^2|\psi(\hat\xi)|\>= \int_{\Pi\times (S\setminus
S_m)}|\Phi(t,x)|^2|\psi(\hat\xi)|d\mu(t,x,\hat\xi),
\end{eqnarray*}
and (\ref{l13}) implies that
$$
\int_{\Pi\times (S\setminus S_m)}|\Phi(t,x)|^2\psi(\hat\xi)d\mu(t,x,\hat\xi)\le C_2\varepsilon\|\Phi\|_2^2
$$
for all $\psi(\hat\xi)\in C_0(S\setminus S_m)$ such that $0\le\psi(\hat\xi)\le 1$. Therefore, we can
claim that
$$
\int_{\Pi\times (S\setminus S_m)}|\Phi(t,x)|^2 d\mu(t,x,\hat\xi)\le C_2\varepsilon\|\Phi\|_2^2,
$$
and since $S\setminus S_0\subset S\setminus S_m$, we obtain the relation
$$
\int_{\Pi\times (S\setminus S_0)}|\Phi(t,x)|^2 d\mu(t,x,\hat\xi)\le C_2\varepsilon\|\Phi\|_2^2,
$$
which holds for arbitrary positive $\varepsilon$. Therefore,
\begin{equation}\label{l14}
\int_{\Pi\times (S\setminus S_0)}|\Phi(t,x)|^2 d\mu(t,x,\hat\xi)=0.
\end{equation}
Since for every point $(t_0,x_0)\in\Pi$ one can find functions $\alpha(t)$, $\beta(x)$ with the
prescribed above properties in such a way that $\Phi(t,x)=\alpha(t)\beta(x)\not=0$ in a neighborhood
of $(t_0,x_0)$, we derive from (\ref{l14}) the desired relation $\mu(\Pi\times (S\setminus S_0))=0$.
\end{proof}

\subsection{Variants of H-measures}

We will need the variant of H-measures with ``continuous indexes'' introduced in
\cite{Pa3}, see also subsequent papers \cite{Pa4,Pa5,PaARMA,PaNHM}.
Let $u_k(x)$ be a bounded sequence in $L^\infty(\Omega)$. Passing to a subsequence if necessary, we can suppose that this sequence converges
to a bounded measure valued function $\nu_x\in\MV(\Omega)$ in the sense of relation (\ref{pr2}). We introduce
the measures $\gamma_x^k(\lambda)= \delta(\lambda-u_k(x))-\nu_x(\lambda)$ and the corresponding
distribution functions $U_k(x,p)=\gamma_x^k((p,+\infty))$, $u_0(x,p)=\nu_x((p,+\infty))$ on
$\Omega\times\R$. Observe that $U_k(x,p), u_0(x,p)\in L^\infty(\Omega)$ for all $p\in\R$, see
\cite[Lemma 2]{Pa3}. We define the set
$$
E=E(\nu_x)=\left\{ \ p_0\in\R \ \mid \
u_0(x,p)\mathop{\to}\limits_{p\to p_0}\, u_0(x,p_0) \ \mbox{ in }
L_{loc}^1(\Omega) \ \right\}.
$$
As was shown in \cite[Lemma 4]{Pa3}, the complement $\R\setminus E$ is at most countable and if
$p\in E$ then $U_k(x,p)\mathop{\rightharpoonup}\limits_{k\to\infty}\, 0$ weakly-$*$ in
$L^\infty(\Omega)$.

The following result was established in  \cite{Pa3} (see also \cite{Pa4,Pa5,PaARMA}~).

\begin{proposition}\label{pro3}
There exist a family of locally finite complex Borel measures
$\{\mu^{pq}\}_{p,q\in E}$ on $\Omega\times S$
and a subsequence
$U_r(x)=\{U_{k_r}(x,p)\}_{p\in E}$
such that for all $\Phi_1(x),\Phi_2(x)\in C_0(\Omega)$ and $\psi(\xi)\in C(S)$
\begin{equation}\label{Hm1}
\<\mu^{pq},\Phi_1(x)\overline{\Phi_2(x)}\psi(\xi)\>= \lim_{r\to\infty}\int_{\R^n}
F(\Phi_1 U_r(\cdot,p))(\xi)\overline{F(\Phi_2 U_r(\cdot,q))(\xi)} \psi(\pi(\xi))d\xi.
\end{equation}
Moreover for every $p_1,\ldots,p_l\in E$, $\zeta_1,\ldots,\zeta_l\in\C$ the measure ${\displaystyle \sum_{i,j=1}^l\zeta_i\overline{\zeta_j}\mu^{p_ip_j}\ge 0}$.
\end{proposition}
Since $|U_r(x,p)|\le 1$, the projection $\pr_\Omega|\mu^{pq}|\le\meas$ and the measures $\mu^{pq}$ admits disintegration:
$\mu^{pq}(x,\xi)=\mu^{pq}_x(\xi)dx$, where $\mu^{pq}_x(\xi)\in M(S)$, $p,q\in E$, $x\in\Omega$, is a family of finite Borel measures on $S$. This means that for every $f(x,\xi)\in C_0(\Omega\times S)$ the function
$$
x\to\langle\mu^{pq}_x(\xi),f(x,\xi)\rangle=\int_{S}f(x,\xi)d\mu^{pq}_x(\xi)
$$
is Lebesgue-measurable, bounded and
\begin{equation}\label{des}
\int_{\Omega\times S}f(x,\xi)d\mu^{pq}(x,\xi)=\int_\Omega\langle\mu^{pq}_x(\xi),f(x,\xi)\rangle dx.
\end{equation}
Let $D\subset E$ be a countable dense subset of $E$. As was demonstrated in \cite{PaNHM}, for almost all $x\in\Omega$ and all $p_1,q_1,p_2,q_2\in D$
$$
\Var(\mu^{p_2q_2}_x-\mu^{p_1q_1}_x)\le
2(\nu_x((p_1,p_2)))^{1/2}+2(\nu_x((q_1,q_2)))^{1/2},
$$
and these estimates implies that the maps $D^2\ni (p,q)\to\mu^{pq}_x$ admit right and left continuous in $M(S)$ extensions $\mu^{pq+}_x$ and $\mu^{pq-}_x$, respectively, on the space of all pairs $(p,q)\in\R^2$, where $x$ runs over some set of full measure in $\Omega$. Actually, these extensions do not depend on the choice of a set $D$. Namely, for different sets $D$ the corresponding measures $\mu^{pq\pm}_x$ coincide for all $x\in\Omega'$, where $\Omega'\subset\Omega$ is a set of full Lebesgue measure, independent of $p,q$. The measures $\mu^{pq+}_x$ inherits the nonnegative definiteness property: for every $p_1,\ldots,p_l\in\R$, $\zeta_1,\ldots,\zeta_l\in\C$ the measures $\displaystyle \sum_{i,j=1}^l\zeta_i\overline{\zeta_j}\mu^{p_ip_j\pm}_x\ge 0$.
As was shown in \cite[Corollary 1]{PaNHM}, this property implies that for each $p,q\in\R$ for every Borel set $A\subset S$
\begin{equation}\label{det}
|\mu^{pq\pm}_x|(A)\le \left(\mu^{pp\pm}_x(A)\mu^{qq\pm}_x(A)\right)^{1/2},
\end{equation}
where $|\mu|$ denotes the variation of a measure $\mu$, i.e., the smallest nonnegative measure $m$ with the property
$|\mu(A)|\le m(A)$ for every Borel set $A$.

The following property connecting the H-measure $\mu_x^{pp\pm}$ and the measure valued function $\nu_x$, corresponding to the same subsequence $u_r$, was proved
in \cite[Corollary~2]{PaNHM}.

\begin{lemma}\label{lem2}
For almost all $x\in\Omega$ (independent of $p$) for all $p$ belonging to the smallest segment $[a, b]$ containing $\supp\nu_x$ the following holds:
If $S_+,S_-\subset S$ are Borel sets such that
$$
\mu^{pp+}_x(S\setminus S_+)=\mu^{pp-}_x(S\setminus S_-)=0
$$
then $S_+\cap S_-\not=\emptyset$. In particular, $\supp\mu^{pp+}_x\cap\supp\mu^{pp-}_x \not=\emptyset$.
\end{lemma}
Denote by $s_{\alpha,\beta}(u)=\max(\alpha,\min(\beta,u))$ the cut-off function, where $\alpha,\beta\in\R$, $\alpha<\beta$. Let $\psi(u)=(\psi_1(u),\ldots,\psi_N(u))\in C(\R,\R^N)$ be a continuous vector-function. Suppose that for every $\alpha,\beta\in\R$, $\alpha<\beta$ the sequences of distributions
$$
\div\psi(s_{\alpha,\beta}(u_k(x))) \ \mbox{ are pre-compact in } H^{-1}_{loc}(\Omega),
$$
where the Sobolev space $H^{-1}_{loc}(\Omega)$ is a locally convex space, consisting of distributions $l\in\D'(\Omega)$ such that for all $f\in C_0^\infty(\Omega)$ the distribution $fl\in H^{-1}=H^{-1}(\R^N)$ (the latter space is dual to the Sobolev space $H^1=W_2^1(\R^N)$), the topology in $H^{-1}_{loc}(\Omega)$ is generated by seminorms $\|fl\|_{H^{-1}}$. Extracting a subsequence if necessary, we can assume that H-measures $\mu^{pq\pm}_x$ are well-defined. Let $H_+(x)$, $H_-(x)$ be the minimal linear subspaces of $\R^N$ containing $\supp\mu^{p_0p_0+}_x$, $\supp\mu^{p_0p_0-}_x$, respectively, where $p_0\in\R$ is fixed.
\begin{proposition}\label{proL2}
For a.e. $x\in\Omega$ for all $\xi\in H_+(x)\cap H_-(x)$ the function
$\xi\cdot\psi(u)$ is constant in some vicinity $|u-p_0|<\delta$ of $p_0$.
\end{proposition}

The statement of Proposition~\ref{proL1} follows from a general result of
\cite[Theorem 3.2]{PaNHM}.

In the case of the sequence $u_k=u(k^2t,k^2\tilde x+k\bar x)$ it follows from Proposition~\ref{pro2} that $\mu^{pp\pm}_{t,x}(S\setminus S_0)=0$. More precisely, we assume
that $u(t,x)\in L^\infty(\Pi)$ is an $x$-periodic function with the standard lattice of periods $\Z^n$.
We consider the sequence $u_k=u(k^2t,k^2\tilde x+k\bar x)$. Passing to a subsequence (and keeping the notation $u_k$ for this subsequence), we may assume that $u_k\to\nu_{t,x}$ in the sense of relation (\ref{pr2}), and that the H-measures $\mu^{pq\pm}_{t,x}$ are defined.

\begin{proposition}\label{proL1} Assume that the family $u(t,\cdot)$, $t>0$, is pre-compact in $L^2(\T^n)$ (possibly, after correction on a set of $t$ of full measure). Then there exists a set $\Pi'\subset\Pi$ of full Lebesgue measure such that for all $(t,x)\in\Pi'$ and all $p,q\in\R$ $\mu^{pq\pm}_{t,x}(S\setminus S_0)=0$.
\end{proposition}

\begin{proof}
Let $M=\|u\|_\infty$. Observe that by Lemma~\ref{lem1} for each $g(\lambda)\in C(\R)$ the weak-$*$ limit of $g(u_k)$ does not depend on $x$:
\begin{equation}\label{14}
g(u_k)\mathop{\rightharpoonup}_{k\to\infty} v_g(t) \ \mbox{ weakly-$*$ in } L^\infty(\Pi).
\end{equation}
Let $G\subset C(\R)$ be a countable dense set. Then the set $A\subset\R_+$ of common Lebesgue points of the functions $v_g(t)$, $g\in G$ has full measure.
Obviously, for each $t\in A$ the functional $g\to v_g(t)$ is uniquely extended to a linear continuous functional on the whole $C(\R)$. Therefore, for some compactly supported measure $\nu_t$
$$
v_g(t)=\langle\nu_t,g(\lambda)\rangle  \quad \forall g(\lambda)\in C(\R).
$$
Since $v_g(t)=0$ whenever $g\equiv 0$ on $[-M,M]$, we see that $\supp\nu_t\subset [-M,M]$. By (\ref{14}) we conclude that $\nu_t$ is a limit measure valued function
of the sequence $u_k$. Hence, the measure valued function $\nu_{t,x}=\nu_t$ does not depend on $x$. By (\ref{Hm1}), (\ref{Hm}) we find that the measure
$\mu^{pp}=\mu^{pp}_{t,x}(\hat\xi)dtdx$ coincides with the H-measure $\mu=\mu(t,x,\hat\xi)$ corresponding to the sequence $U_k(t,x,p)=U(k^2t,k^2\tilde x+k\bar x,p)-v(t,p)$, where
$U(t,x,p)=\theta(u(t,x)-p)$ ($\theta(r)$ being the Heaviside function), $v(t,p)=\nu_t((p,+\infty))$ is a weak-$*$ limit of the sequence $U_k(t,x,p)$.
By Proposition~\ref{pro2} we claim that $\mu(\Pi\times(S\setminus S_0))=0$. From (\ref{des}) it follows that for all $f(t,x)\in C_0(\Pi)$, $g(\hat\xi)\in C(S)$
\begin{equation}\label{des1}
\int_\Pi \langle\mu^{pp}_{t,x}(\hat\xi),g(\hat\xi)\rangle f(t,x)dtdx=\int_{\Pi\times S}f(t,x)g(\hat\xi)d\mu(t,x,\hat\xi).
\end{equation}
Obviously, (\ref{des1}) remains valid for Borel functions $g(\hat\xi)$. Taking $g(\hat\xi)$ being the indicator function of the set $S\setminus S_0$, we arrive at the relation
$$
\int_\Pi \langle\mu^{pp}_{t,x}(\hat\xi),g(\hat\xi)\rangle f(t,x)dtdx=0 \quad \forall f(t,x)\in C_0(\Pi).
$$
Therefore, $\langle\mu^{pp}_{t,x}(\hat\xi),g(\hat\xi)\rangle=0$ for almost all $(t,x)\in\Pi$. Now we fix some countable dense set $D\subset E$ and choose a set of full Lebesgue measure $\Pi_0\subset\Pi$ such that $\langle\mu^{pp}_{t,x}(\hat\xi),g(\hat\xi)\rangle=0$ for all $(t,x)\in\Pi_0$ and $p\in D$. By the definition of measures $\mu^{pp+}_{t,x}$, $\mu^{pp-}_{t,x}$, they are, respectively, right- and left-continuous in $M(S)$ with respect to $p\in\R$, for some set $\Pi'\subset\Pi_0$ of full measure of values $(t,x)$. Therefore, $\langle\mu^{pp\pm}_{t,x}(\hat\xi),g(\hat\xi)\rangle=0$, that is, $\mu^{pp\pm}_{t,x}(S\setminus S_0)=0$ for all $p\in\R$, $(t,x)\in\Pi'$. To complete the proof, it only remains to apply relation (\ref{det}).
\end{proof}

\section{The proof of the decay property}\label{secmain}
We assume that $u(t,x)$ is an $x$-periodic e.s. of (\ref{1}), (\ref{2}). As was demonstrated in section~\ref{sec3}, we may suppose that conditions (R1)--(R4) hold.
In view of (\ref{after}) for every $k\in\N$ the function $u_k=u(k^2t,k^2\tilde x+k\bar x)$ is an e.s. of the
equation
\begin{equation}
\label{1k}
u_t+\div_x (\varphi^k(u)-a^k(u)\nabla_x u)=0,
\end{equation}
where $\varphi^k(u)$ is the vector with components $\varphi^k_i(u)=\varphi_i(u)$, $i=1,\ldots,d$, $\varphi^k_i(u)=k\varphi_i(u)$, $i=d+1,\ldots,n$, while the symmetric matrix $a^k(u)$ has the entries $(a^k(u))_{ij}=\epsilon_{ki}\epsilon_{kj}a_{ij}(u)$, $1\le i,j\le n$, where
$\epsilon_{ki}=k^{-1}$ for $1\le i\le d$, $\epsilon_{ki}=1$ for $d<i\le n$.

If $a(u)=b(u)^\top b(u)$ is an admissible representation then
$a^k(u)=b_k(u)^\top b_k(u)$, $(b_k)_{rj}=\epsilon_{kj} b_{rj}$, is an admissible representation for the matrix $a^k(u)$. Let $\alpha,\beta\in\R$, $\alpha<\beta$, $s_{\alpha,\beta}(u)=\max(\alpha,\min(\beta,u))$ be the corresponding cut-off function, $A^k(u)$ be a primitive of the matrix $a^k(u)$. By  Definition~\ref{def1}(i,ii) and Remark~\ref{rem0}
\begin{eqnarray}\label{15}
\sum_{j=1}^n [(A^k(s_{\alpha,\beta}(u_k)))_{ij}]_{x_j}=\sum_{j=1}^n \sum_{r=1}^n (T_{(b_k)_{ri}(u)\chi_{\alpha,\beta}(u)}((B_k)_{rj})(u_k))_{x_j} \nonumber\\
=\sum_{r=1}^n (b_k)_{ri}(u_k)\chi_{\alpha,\beta}(u_k)\div_x (B_k)_r(u_k),
\end{eqnarray}
where the matrix $B_k$ is a primitive of $b_k$, $(B_k)'(u)=b_k(u)$, and $(B_k)_r$
is the vector with components $(B_k)_{rj}(u)$, $j=1,\ldots,n$. By $\chi_{\alpha,\beta}(u)$ we denote the indicator function of the interval $(\alpha,\beta)$. By Lemma~\ref{lemA2}
$\div_x (B_k)_r(u_k)\in L^2_{loc}(\Pi)$, and
\begin{eqnarray}\label{16}
\int_{(\delta,+\infty)\times\T^n}\sum_{r=1}^n |\div_x (B_k)_r(u_k)|^2dtdx\le \nonumber\\
\frac{1}{2}\int_{\T^n} |u_k(\delta,x)|^2dx-\esslim_{T\to+\infty}\frac{1}{2}\int_{\T^n} |u_k(T,x)|^2dx,
\end{eqnarray}
where $\delta>0$ is a common Lebesgue point of the functions $\displaystyle t\to\int_{\T^n} |u_k(t,x)|^2dx$, $k\in\N$. Since the maps $x\to y=k^2\tilde x+k\bar x$ keep the Lebesgue measure on the torus $\T^n$, then for each $t>0$
$$
\int_{\T^n} |u_k(t,x)|^2dx=\int_{\T^n} |u(k^2t,k^2\tilde x+k\bar x)|^2dx=\int_{\T^n} |u(k^2t,y)|^2dy,
$$
and it follows from (\ref{16}) that for almost each $\delta>0$
\begin{eqnarray*}
\int_{(\delta,+\infty)\times\T^n}\sum_{r=1}^n |\div_x (B_k)_r(u_k)|^2dtdx\le \\
\frac{1}{2}\int_{\T^n} |u(k^2\delta,y)|^2dy-\esslim_{T\to+\infty}\frac{1}{2}\int_{\T^n} |u(T,y)|^2dy\mathop{\to}_{k\to\infty} 0,
\end{eqnarray*}
which implies that for all $r=1,\ldots,n$
\begin{equation}\label{17}
\div_x (B_k)_r(u_k)\to 0 \ \mbox{ in } L^2_{loc}(\Pi)
\end{equation}
as $k\to\infty$.
It follows from (\ref{15}) that for $1\le i\le n$ the distributions
$$
\sum_{j=1}^n [(A_k(s_{\alpha,\beta}(u_k)))_{ij}]_{x_j}\to 0 \ \mbox{ in } L^2_{loc}(\Pi)
$$
This implies that
\begin{equation}\label{18}
\sum_{i=1}^n\sum_{j=1}^n [(A^k(s_{\alpha,\beta}(u_k)))_{ij}]_{x_ix_j}\to 0
\end{equation}
as $k\to\infty$ in the Sobolev space $H^{-1}_{loc}(\Pi)$.

Observe that for every $g(u)\in C(\R)$
\begin{eqnarray*}
g(s_{\alpha,\beta}(u))=\sgn^+(u-\alpha)(g(u)-g(\alpha))-\sgn^+(u-\beta)(g(u)-g(\beta))+ \\ g(\alpha)= T_{\sgn^+(u-\alpha)}(g)(u)-T_{\sgn^+(u-\beta)}(g)(u)+\const,
\end{eqnarray*}
where $\sgn^+ u=(\max(u,0))'$ is the Heaviside function. Using this relation and Lemma~\ref{lemA1}, we obtain
\begin{equation}\label{19}
(s_{\alpha,\beta}(u_k))_t+\div_x \varphi^k(s_{\alpha,\beta}(u_k))-D^2\cdot A^k(s_{\alpha,\beta}(u_k))=\mu^k_\beta-\mu^k_\alpha
\end{equation}
in $\D'(\Pi)$, where $\mu^k_r$, $r\in\R$ are nonnegative measures defined by (\ref{distr1}) with $\varphi(u)=\varphi^k(u)$, $A(u)=A^k(u)$, $u=u_k$, and $\eta(u)=(u-r)^+\doteq\max(u-r,0)$. By Corollary~\ref{cor0}(iii)
$$
\mu_r^k(\R_+\times\T^n)\le\int_{\T^n}(u_0(k^2\tilde x+k\bar x)-r)^+dx=\int_{\T^n}(u_0(x)-r)^+dx,
$$
and the sequence $\mu^k_\beta-\mu^k_\alpha$ is bounded in the space of measures $M_{loc}(\Pi)$ equipped with standard locally convex topology generated by semi-norms
$p_K(\mu)=|\mu|(K)$, where $K$ is an arbitrary compact subset of $\Pi$. By Murat's interpolation lemma \cite{Mu}, it follows from (\ref{19}) and (\ref{18}) that the sequence of distributions
\begin{equation}\label{20}
(s_{\alpha,\beta}(u_k))_t+\div_x \varphi^k(s_{\alpha,\beta}(u_k))
\end{equation}
is pre-compact in $H^{-1}_{loc}(\Pi)$.

By requirement (ii) of Definition~\ref{def1}
$$
\div_x (B_k)_r(s_{\alpha,\beta}(u_k))=\chi_{\alpha,\beta}(u_k)\div_x (B_k)_r(u_k) \to 0
$$
as $k\to\infty$ in $L^2_{loc}(\Pi)$ and therefore in $H^{-1}_{loc}(\Pi)$ as well. Since
$$
\sum_{j=d+1}^n B_{rj}(s_{\alpha,\beta}(u_k))_{x_j}=\div_x (B_k)_r(s_{\alpha,\beta}(u_k))-\frac{1}{k}\sum_{j=1}^d B_{rj}(s_{\alpha,\beta}(u_k))_{x_j},
$$
we claim that the sequences of distributions
\begin{equation}\label{21}
\sum_{j=d+1}^n B_{rj}(s_{\alpha,\beta}(u_k))_{x_j}\to 0
\end{equation}
in $H^{-1}_{loc}(\Pi)$ for all $r=1,\ldots,n$. In particular, these sequences are pre-compact in $H^{-1}_{loc}(\Pi)$.

Taking into account Remark~\ref{rem3} and Corollary~\ref{cor2}, we conclude that the sequence $u_k$ converges as $k\to\infty$ to a constant measure valued function $\nu_{t,x}\equiv\nu$ (in the sense of relation (\ref{pr2})~). Let $[a_0,b_0]$ be the smallest segment containing $\supp\nu$. We suppose that $a_0<b_0$ and are going to get a contradiction.

Notice that by Corollary~\ref{cor0}(ii) for a.e. $t>0$
$\displaystyle
\int_{\T^n}u(t,x)dx=\int_{\T^n} u_0(x)dx=I.
$
By Lemma~\ref{lem1} we find that $u_k\rightharpoonup I$ as $k\to\infty$ weakly-$*$ in $L^\infty(\Pi)$. In view of relation (\ref{pr2}) $I=\int\lambda d\nu(\lambda)\in (a_0,b_0)$.
By our assumption (R3) the functions $\varphi_i(u)$, $i=d+1,\ldots,n$, are constant on some interval $(a_1,b_1)\ni I$. Without loss of generality we can assume that
$(a_1,b_1)\subset (a_0,b_0)$.

\begin{lemma}\label{lem3}
There exists an interval $(a_2,b_2)\subset (a_1,b_1)$ such that
\begin{equation}\label{22}
\int s_{a_2,b_2}(\lambda)d\nu(\lambda)=I.
\end{equation}
\end{lemma}

\begin{proof}
Denote $I_1=\int s_{a_1,b_1}(\lambda)d\nu(\lambda)$. If $I_1=I$, there is nothing to prove, we set $(a_2,b_2)=(a_1,b_1)$. If $I_1<I$, then the continuous function
$f(a)=\int s_{a,b_1}(\lambda)d\nu(\lambda)$ takes values $f(a_1)=I_1<I$, $f(I)>I$. Therefore, there exists $a_2\in (a_1,I)$ such that $f(a_2)=I$. Taking $b_2=b_1$, we conclude that
(\ref{22}) holds. In the case $I_1>I$ we consider the continuous function $g(b)=\int s_{a_1,b}(\lambda)d\nu(\lambda)$ and observe that $g(I)<I$, $g(b_1)=I_1>I$. Therefore, there is a value $b_2\in (I,b_1)$ such that $g(b_2)=I$ and (\ref{22}) follows with $a_2=a_1$.
\end{proof}

Let $(a_2,b_2)$ be an interval indicated in Lemma~\ref{lem3}. We consider the sequence $v_k=s_{a_2,b_2}(u_k)$. In correspondence with (\ref{pr2})
$$
v_k\rightharpoonup \int s_{a_2,b_2}(\lambda)d\nu(\lambda)=I.
$$
Observe that in the case when $(\alpha,\beta)\cap (a_2,b_2)=(\alpha_1,\beta_1)\not=\emptyset$
$$s_{\alpha,\beta}(v_k)=s_{\alpha_1,\beta_1}(u_k).$$
Otherwise, $s_{\alpha,\beta}(v_k)\equiv c=\const$.
Since the flux functions $\varphi_i(u)$, $i=d+1,\ldots,n$, are constant on the segment $[a_2,b_2]$, we conclude that $\varphi_i(s_{\alpha,\beta}(v_k))_{x_i}=0$ and in view of (\ref{20}) for all $\alpha,\beta\in\R$, $\alpha<\beta$ the sequences of distributions
\begin{equation}\label{23}
(s_{\alpha,\beta}(v_k))_t+\sum_{i=1}^d \varphi_i(s_{\alpha,\beta}(v_k))_{x_i}
\end{equation}
are pre-compact in $H^{-1}_{loc}(\Pi)$. Similarly, relation (\ref{21}) implies that the sequences
\begin{equation}\label{21a}
\sum_{j=d+1}^n B_{rj}(s_{\alpha,\beta}(v_k))_{x_j}
\end{equation}
are pre-compact in $H^{-1}_{loc}(\Pi)$ as well.

 As follows from relation (\ref{pr2}), the limit measure valued function for the sequence $v_k$ is a constant measure valued function $\nu_{t,x}\equiv\nu_1=s_{a_2,b_2}^*\nu$ is push-forward measure of $\nu$ under the map $s_{a_2,b_2}(\lambda)$. Since $[a_2,b_2]\subset [a_0,b_0]$, $[a_2,b_2]$ is the minimal segment containing $\supp\nu_1$. Passing to a subsequence if necessary, we can consider H-measures $\mu^{pq\pm}_{t,x}$ corresponding to the sequence $v_k$. Let $p\in (a_2,b_2)$,
$H_+(t,x)$, $H_-(t,x)$ be the linear spans of supports $\supp\mu^{pp+}_x$ and $\supp\mu^{pp-}_x$, respectively. As follows from Proposition~\ref{proL2} and pre-compactness of sequences (\ref{23}), (\ref{21a}), for a.e. $(t,x)\in\Pi$ for all $(\tau,\xi)\in H_+(t,x)\cap H_-(t,x)$ the functions
\begin{eqnarray}\label{24}
u\to \tau u+\varphi(u)\cdot\tilde\xi= \tau u+\sum_{j=1}^d\varphi_j(u)\xi_j=\const; \\
\label{25}
u\to B_r(u)\cdot\bar\xi=\sum_{j=d+1}^n B_{rj}(u)\xi_j=\const,  \quad r=1,\ldots,n,
\end{eqnarray}
in some vicinity of $p$. For fixed $(t,x)$ we denote $H_\pm=H_\pm(t,x)$ and introduce the sets $S_\pm=H_\pm\cap S_0$. Since $S\setminus S_\pm\subset (S\setminus H_\pm)\cap (S\setminus S_0)$ then
$$
\mu^{pp\pm}_{t,x}(S\setminus S_\pm)\le \mu^{pp\pm}_{t,x}(S\setminus H_\pm)+\mu^{pp\pm}_{t,x}(S\setminus S_0)=0,
$$
where we use Proposition~\ref{proL1}.
We remark that by Corollary~\ref{cor1} the family $u(t,\cdot)$, $t>0$ is pre-compact in $L^2(\T^n)$, and the conditions of Proposition~\ref{proL1} are actually satisfied.

Applying Lemma~\ref{lem2}, we obtain that for a.e. $(t,x)\in\Pi$ the set $H_+(t,x)\cap H_-(t,x)\cap S_0$ is not empty and for all $(\tau,\xi)\in H_+(t,x)\cap H_-(t,x)\cap S_0$
identities (\ref{24}), (\ref{25}) hold. In particular, we can take $p=I\in (a_2,b_2)$. Then it follows from (\ref{24}) and assumption (R2) that $\tilde\xi=0$, $\tau=0$. Further,
in view of (\ref{25}) $b(u)\bar\xi=\frac{d}{du}B(u)\bar\xi=0$ in some vicinity of $I$. This implies that $a(u)\bar\xi=b(u)^\perp b(u)\bar\xi=0$ a.e. in this vicinity. By assumption (R4), we claim that $\bar\xi=0$. Hence, $\hat\xi=(\tau,\xi)=0$, which contradicts to the condition $\hat\xi\in S$. We conclude that $a_0=b_0=I$, that is, $\nu=\delta(\lambda-I)$. This means that the limit measure valued function $\nu_{t,x}(\lambda)\equiv\nu(\lambda)=\delta(\lambda-I)$ of the sequence $u_k$ is regular, and by Theorem~\ref{thT}, $u_k\to I$ as $k\to\infty$ in $L^1_{loc}(\Pi)$. This implies that for a.e. $t>0$ $u_k(t,\cdot)\to I$ as $k\to\infty$ in $L^1(\T^n)$.  We fix such $t=t_0$. Then
\begin{eqnarray}\label{26}
\int_{\T^n}|u(k^2t_0,x)-I|dx=\int_{\T^n} |u(k^2t_0,k^2\tilde x+k\bar x)-I|dx=\nonumber\\ \int_{\T^n} |u_k(t_0,x)-I|dx\mathop{\to}_{k\to\infty} 0.
\end{eqnarray}
For $t>k^2t_0$ we use Lemma~\ref{lemA3} for e.s. $u_1=u$, $u_2\equiv I$ and find that for almost each such $t$
$$
\int_{\T^n}|u(t,x)-I|dx\le\int_{\T^n} |u(k^2t_0,x)-I|dx,
$$
which, together with (\ref{26}), implies the desired decay relation (\ref{dec}).

\medskip
\begin{remark}
Using methods of paper \cite{PaJHDE}, we can extend our results to the case of almost periodic (in the Besicovitch sense) initial function $u_0(x)$.
Let
$$ C_R=\{ \ x=(x_1,\ldots,x_n)\in\R^n \ | \ |x|_\infty=\max_{i=1,\ldots,n}|x_i|\le R/2 \ \}, \quad R>0,$$

$$N_1(u)=\limsup_{R\to +\infty} R^{-n}\int_{C_R} |u(x)|dx
$$
be the mean $L^1$-norm of a function $u(x)\in L^1_{loc}(\R^n)$.
Recall, (~see \cite{Bes}~) that Besicovitch space $\B^1(\R^n)$ is the closure of trigonometric polynomials, i.e., finite sums $\sum a_\lambda e^{2\pi i\lambda\cdot x}$, with ${i^2=-1}$, $\lambda\in\R^n$, in the quotient space $B^1(\R^n)/B^1_0(\R^n)$, where
$$
B^1(\R^n)=\{ u\in L^1_{loc}(\R^n) \ | \ N_1(u)<+\infty \}, \ B^1_0(\R^n)=\{ u\in L^1_{loc}(\R^n) \ | \ N_1(u)=0 \}.
$$
The space $\B^1(\R^n)$ is a Banach space equipped with the norm $\|u\|_1=N_1(u)$.

It is known \cite{Bes} that for each $u\in \B^1(\R^n)$ there exist the mean value
$$\dashint_{\R^n} u(x)dx\doteq\lim\limits_{R\to+\infty}R^{-n}\int_{C_R} u(x)dx$$ and, more generally,  the Bohr-Fourier coefficients
$$
a_\lambda=\dashint_{\R^n} u(x)e^{-2\pi i\lambda\cdot x}dx, \quad\lambda\in\R^n.
$$
The set
$$ Sp(u)=\{ \ \lambda\in\R^n \ | \ a_\lambda\not=0 \ \} $$ is at most countable and is called the spectrum of an almost periodic function $u$.
We denote by $M(u)$ the smallest additive subgroup of $\R^n$ containing $Sp(u)$.

Suppose that $u_0\in\B^1(\R^n)\cap L^\infty(\R^n)$, $\displaystyle I=\dashint_{\R^n} u_0(x)dx$, $M_0=M(u_0)$, and that $u=u(t,x)$ is an e.s. of (\ref{1}), (\ref{2}).
Arguing as in \cite{PaJHDE}, we may conclude that $u(t,\cdot)\in\B^1(\R^n)$ and that $M(u(t,\cdot))\subset M_0$ for a.e. $t>0$. The decay property is modified as follows.
\begin{theorem}\label{thM1}
Assume that  for all $\xi\in M_0$, $\xi\not=0$ there is no vicinity of $I$, where simultaneously the function $\xi\cdot\varphi(u)$ is affine and the function $a(u)\xi\cdot\xi\equiv 0$. Then
$$
\esslim_{t\to+\infty} u(t,x)=I \ \mbox{ in } \B^1(\R^n).
$$
\end{theorem}
\end{remark}

\section{Acknowledgements}
This work was supported by the Ministry of Education and Science of the Russian  Federation (project no. 1.445.2016/1.4) and by the Russian Foundation for Basic Research (grant 18-01-00258-a.)

\end{document}